\documentclass{amsart}

 \usepackage{latexsym}

 \usepackage{amssymb}

 \usepackage{amsfonts}

 \usepackage{amsmath}

 \usepackage{amsthm}
\usepackage{float}

\usepackage{graphics}

\usepackage[pdftex]{graphicx}
\newtheorem*{thmA}{Theorem A}
\newtheorem*{thmB}{Theorem B}
\newtheorem*{thmD}{Theorem D}
\newtheorem*{corC}{Corollary C}
\newtheorem{theorem}{Theorem}[section]
\newtheorem{corollary}[theorem]{Corollary}
\newtheorem{lemma}[theorem]{Lemma}
\newtheorem{proposition}[theorem]{Proposition}
\newtheorem{example}[theorem]{Example}
\newtheorem{examples}[theorem]{Examples}
\newtheorem{remark}[theorem]{Remark}

\def\Si{\Sigma}
\def\la{\langle}
\def\ra{\rangle}
\def\Sm0#1{{{\rm GL}}(1,#1)}

\begin{document}
\title[Covering symmetric alternating]
{Normal coverings of finite symmetric and alternating groups}
\author[D. Bubboloni]{Daniela Bubboloni}
\address{D. Bubboloni, Dipartimento di Matematica per le
  Decisioni\newline
Universit\`a di Firenze, via Lombroso 6/17\\
50134 Firenze, Italy.}
\email{daniela.bubboloni@dmd.unifi.it}

\author[C.E. Praeger]{Cheryl E. Praeger}
\address{C. E. Praeger,  Centre for the Mathematics of Symmetry and Computation\\
School of Mathematics and Statistics\\
The University of Western Australia\\
35 Stirling Highway \\ Crawley, WA 6009, Australia}
\email{cheryl.praeger@uwa.edu.au}

\begin{abstract}  In this paper we investigate the minimum number of maximal subgroups $H_i,\ i=1 \dots k $ of the symmetric group $S_n$ (or the alternating group $A_n$) such that each element in the group $S_n$ (respectively $A_n$) lies in some conjugate of one of the $H_i.$ We prove that this number lies between $a\phi(n)$ and $bn$ for certain constants $a, b$, where $\phi(n)$ is the Euler phi-function, and  we show that the number depends on the arithmetical complexity of $n$. Moreover in the case where $n$ is divisible by at most two primes, we obtain an upper bound of $2+\phi(n)/2$, and we determine the exact value for  $S_n$ when $n$ is odd and for $A_n$ when $n$ is even.
\bigskip


\end{abstract}
\maketitle
\section{Introduction}
Let $G$ be a finite group.
A {\em covering} of $G$ is a set of proper subgroups of $G,$ called {\em components}, whose union is $G.$ The smallest integer $m$ such that $G$ has a covering of cardinality $m$ is denoted $\sigma(G)$, and was introduced by J. H. E. Cohn \cite{CO}. A  covering of cardinality $\sigma(G)$ is called a {\em minimal covering}. % which realizes $\sigma(G)$.

In \cite{MA}, A. Mar\'{o}ti  found ``exact or asymptotic formulas'' for $\sigma(S_n)$ and $\sigma(A_n)$ and, for example, proved that $\sigma(S_n)=2^{n-1}$ for odd $n\neq 9$. In many cases the  minimal
coverings which realize $\sigma(G)$ are sets of conjugacy classes of subgroups. For instance, this is the case when $G=S_n$ for odd integers $n\neq 9$  and when $G=A_n$ for even integers $n$ not divisible by $4.$ In this paper we study coverings which are unions of conjugacy classes of subgroups with a view to minimising the number of classes involved. Part of our motivation comes from a number theoretic application which we discuss in Subsection~\ref{sub:appn}.

We define a {\em normal covering} for a finite group $G$ as a covering which is invariant under $G$-conjugation. Such a covering may be viewed as a generalization of the notion of a normal partition of a group introduced by R. Baer \cite{B}. If  $\  H_1,\dots,H_k$ are pairwise non-conjugate proper subgroups of $G$ such that
$$
G=\bigcup_ {i=1}^k\bigcup_{g\in G}H_i^g
$$
then we call the $H_i$ {\em basic components} of the corresponding normal covering $\Si=\{H_i^g\,|\,1\leq i\leq k, g\in G\}$, and we say that $\delta=\{H_1,\dots,H_k\}$ is a {\em basic set} for $G$ which generates %the  $k${\em -normal covering}
$\Sigma.$ The smallest size of a basic set is denoted $\gamma (G)$,
%called the \emph{minimal normal covering number} of $G$, and
a normal covering with $\gamma (G)$ basic components is called a {\em minimal normal covering}, and a basic set of size $\gamma (G)$ is called a {\em minimal basic set} of $G$. % which realizes $\gamma(G).$\\

A finite group admits a normal covering if and only if it is non-cyclic. Moreover, for a non-cyclic group $G$, the set of all  maximal subgroups of $G$ forms a normal covering; it is well known that $\gamma(G)\geq 2$, and this lower bound can be attained (see \cite{BBH}), while if $G$ is abelian then $\gamma(G)=\sigma(G)\geq 3.$ We can always replace a normal covering by one with the same number of basic components in which each basic component is a maximal subgroup.

In this paper we investigate $\gamma (G)$ where $G$ is the symmetric or  alternating group acting on the set $\Omega=\{1,\dots,n\}$. As well as proving upper and lower bounds for $\gamma(G)$ for all values of $n$, we construct
some remarkable minimal normal coverings. Our main results are given in Subsection~\ref{sub:main}.

In \cite{BBH}, the first author showed that $\gamma (S_n)=2$ if and only if $3\leq n\leq 6$ and observed that with a more tedious subcases analysis it is proved also that $\gamma (A_n)=2$ if and only if $4\leq n\leq 8.$ Details for this last result can be found in \cite{BU}.

Some of the ideas in  \cite{MA} are of interest for us, but our methods are different from those in \cite{MA}. Consider for example the case of $S_7.$ This group has no normal covering with only two basic components, but admits a
normal covering with basic set
$$
\delta=\{ S_2\times S_5,\quad S_3\times S_4,\quad AGL_1(7)\cong C_7\rtimes C_6\}
$$
consisting of maximal subgroups (for it is easily checked that each type of permutation is present in at least  one component). Thus  $\gamma(S_7)=3$ and $\delta$ generates a minimal normal covering of $S_7$. On the other hand, this normal covering has cardinality $120$, which is greater than $\sigma(S_7)=2^6$, so as a covering it is not minimal. On the other hand, a minimal covering $\Delta$  of $S_7$ is given by  the four subgroups
$$
S_2\times S_5,\quad S_3\times S_4,\quad S_6,\quad A_7
$$
together with all their conjugates in $S_7$, so that $\sigma(S_7)$ is realizable through a normal covering with four basic components.
\emph{Thus, even if a minimal covering is normal, it is not necessarily a minimal normal covering; and a minimal normal covering is not necessarily a minimal covering.}
\medskip

\subsection{An application in Galois theory}\label{sub:appn}
It is worth noting that normal coverings shed light on some questions from number theory. Let $f(x)\in \mathbb{Z}[x]$ be a polynomial which has a root $mod\,p,$ for all primes $p$ and consider its Galois group over the rationals $G=Gal_{\mathbb{Q}}(f)$ acting on the set $\Omega$ of roots of $f.$ Let $f_1(x),\dots, f_k(x)\in \mathbb{Z}[x]$ be the
distinct irreducible factors of $f(x)$ over $\mathbb{Q}$, and suppose that no $f_i$ is linear. %Observe that $\Omega$ is partitioned into the sets $\Omega_i$ of roots of $f_i$.
Choose a root $\omega_i$ of $f_i$ for every $i=1,\dots,k$, and consider the stabilizer $G_{\omega_i}$ of $\omega_i$ in $G.$ Then by \cite[Theorem 2]{BB},
$$
G=\bigcup_{g\in G} \bigcup_{i=1}^k G_{\omega_i} ^g
$$
and since $G_{\omega_i}< G$ (as no $f_i$ is linear), we obtain a normal covering for $G$ with $k$ basic components. It follows that $k\geq \gamma (G)$. In other words, for a polynomial $f(x)$ which has a root $mod\,p,$ for all primes $p$, but no root in $\mathbb{Q}$, the number of basic components in a minimal normal covering of its Galois group is a lower bound for the number of distinct irreducible factors of $f(x)$ over $\mathbb{Q}$. In this context the significance of our results relies on the fact that the most common Galois groups are the symmetric and alternating groups (see \cite{WA}). In general when a normal covering for a group $G$ is known we may ask if $G$ is the Galois group (over the rationals) of a polynomial which has a root $mod\,p,$ for all $p$ and which splits into $\gamma(G)$ distinct irreducible factors.
This variant of the inverse problem in Galois theory has been studied by J. Sonn in \cite{SO}. Recently D. Rabayev and J. Sonn \cite{RS} proved that the minimal normal coverings of size 2 of $A_n$ and $S_n$, in \cite{BBH} and \cite{BU}, all lead to realizations of these groups as Galois groups.

\subsection{Main results}\label{sub:main}

First we state our general upper bounds. These are proved with more detailed information in Proposition~\ref{upper} and Proposition ~\ref{sym prime}.
\begin{thmA} For $G=S_n$ or $A_n$ with $n\geq 4$,
{\openup 3pt
\[
\gamma(G)\leq\left\{\begin{array}{ll}
\left\lfloor
\frac{n+4}{4}\right\rfloor & \mbox{if $G=S_n$ or $A_n$ with $n$ even}\\
\frac{n-1}{2}              & \mbox{if $G=S_n$ with $n$ odd}\\
\left\lfloor\frac{n+3}{3}\right\rfloor & \mbox{if $G=A_n$ with $n$ odd.}\\
                    \end{array}\right.
\]
}
\end{thmA}
Next we give our general lower bounds. The statement uses the function $\phi(I;n)$, for an interval $I\subseteq [0,n]$, which is defined as the number of positive integers $i\in I$ such that $(i,n)=1$ (see Lemma~\ref{asymptoticI}). We also use the notation $f\,\sim\,g$ to denote that `$f$ is asymptotic to $g$' (see vi) in Section~\ref{basic number}). These lower bounds follow from the results proved in Propositions~\ref{gamma twocycles}, \ref{even-sym}, and Corollary~\ref{cor-lower}, with reference to \cite{BBH,BU} for $n\leq8$.

\begin{thmB} Let  $n\geq 5$. Then for $G=A_n$, or for  $G=S_n$ with $n$ odd,
{\openup 2pt
\[
\gamma(G)\geq\left\{\begin{array}{ll}
\frac{\phi(n)}{2}+1 & \mbox{if $G=S_n$ with $n$ odd and composite,}\\
                    & \mbox{\ or $G=A_n$ with $n$ even and $n\ne8$}\\
\frac{\phi(n)}{2}              & \mbox{if $G=S_n$ with $n$ prime, or $G=A_8$}\\
\frac{\phi(n)+2}{4}              & \mbox{if $G=A_n$ with $n$ odd and $n\ne 2^e+1$}\\
\frac{\phi(n)}{4}              & \mbox{if $G=A_n$ with $n=2^e+1$ for some $e$}\\
                    \end{array}\right.
\]
}
while for $S_n$ with $n$ even,
$$
\gamma(S_n)\geq \frac{\phi(I(n);n)}{c(n)}+1\,\sim\,\frac{\phi(n)}{c(n)d(n)}
$$
where $I(n), c(n), d(n)$ are as in one of the lines of Table~{\rm\ref{tbl1}}.
\end{thmB}

{\openup 2pt
\begin{center}
 \begin{table}[H]
\begin{tabular}{llll}\hline
$n$ even&$I(n)$&$c(n)$&$d(n)$\\ \hline \hline
$3\,\nmid\,n$&$[1,\frac{n-1}{3})$&2&3\\
$n\equiv 3$ or $6\pmod{9}$&$[1,\frac{n}{9})$&1&9\\
$9\,\mid \,n$&$[\frac{n}{18},\frac{n}{6})$&1&9\\  \hline
\end{tabular}
\caption{Interval $I(n)$ and constants $c(n), d(n)$ for Theorem~B}\label{tbl1}
 \end{table}
\end{center}
}

We note that these two theorems show that there is a function $f(n)$, namely the minimum of the various functions occurring in Theorems A and B, such that $f(n)\,\sim\,\phi(n)$, and such that
\[
\frac{f(n)}{9}\leq \gamma(S_n)\leq\frac{n-1}{2}, \quad\mbox{and}\quad \frac{f(n)}{4}\leq \gamma(A_n)\leq \frac{n+3}{3}.
\]
Then, since $\ \displaystyle{\lim_{n\rightarrow +\infty}\phi(n)=+\infty}$, we see that the gamma functions are unbounded. Moreover, surprisingly, also the difference between $\gamma(S_n)$ and $\gamma(A_n)$ is unbounded (see Corollary \ref{distance}).

\begin{corC}\label{limits}\begin{description}\item[(a)] $ \displaystyle{\lim_{n\rightarrow +\infty}\gamma(S_n)}= \displaystyle{\lim_{n\rightarrow +\infty}\gamma(A_n)=+\infty}.$\\ In particular, for any $k\in \mathbf{N}$ with $k\geq 2$, and $G=S_n,\ A_n,$ the equation $\gamma(G)=k$ has only a finite number of solutions.
\item[(b)] $\displaystyle{\limsup_{n\rightarrow +\infty}\, [\gamma(S_n)-\gamma(A_n)]=+\infty.}$
\end{description}
\end{corC}

Finally we summarise our results giving considerable improvements to the general upper bounds, and in many cases exact values, for the $\gamma$-function when $n$ has at most two prime divisors. These results follow from  Proposition~\ref{sym prime} (for $S_p$), Proposition~\ref{prime power} (for $n=p^\alpha\geq p^2$), Proposition~\ref{two primes} (for $n=pq$), and Proposition~\ref{two primes power}.

\begin{thmD} %[Alternate version]
Let $p, q$ be primes such that $2\leq p<q$, let $\alpha, \beta$ be positive integers, let $n$ and $\delta(n)$ satisfy one of the lines of Table ~{\rm \ref{tb1:delta}}, and let  $G=A_n$ or $S_n$. Then
\[
\gamma(G)\leq \frac{\phi(n)}{2} +\delta(n)
\]
and equality holds for $G=S_n$ with $n$ odd, and for $G=A_n$ with $n$ even.
{\openup 2pt
\begin{center}
 \begin{table}[H]
\begin{tabular}{lll}\hline
$n$&$\delta(n)$&Conditions on $p,\alpha,\beta,G$\\ \hline
$p$& $0$         &$p\geq5,\ G=S_p$\\ %\hline
$p^\alpha$&$1$&$p$ odd, $\alpha\geq2,\ G=S_{p^{\alpha}};$ $p=2$,
$\alpha\geq4$ \\
$pq$ & $1$&\\
$p^\alpha q^\beta$&$2$&$\alpha+\beta\geq3,\ \alpha\geq1,\ \beta\geq 1$\\  \hline
\end{tabular}
\caption{Values of $n$ and $\delta(n)$ for Theorem D.}\label{tb1:delta}
 \end{table}
\end{center}
}
\end{thmD}
Propositions~\ref{sym prime} and~\ref{prime power} also give the same upper bounds for $\gamma(A_n)$ as for $\gamma(S_n)$, in the case where $n=p^\alpha\geq p$ is odd. However in these cases the upper bounds are weaker than the general upper bounds in Theorem A, so we do not include them in the statement of Theorem D. For the other odd values of $n$, namely if $n=p^\alpha q^\beta$ is odd with both $\alpha\geq 1,\beta\geq1$, then the upper bound for $\gamma(A_n)$ given by Theorem D is better when $p^\alpha=3$ (for any $q^\beta$) and, for example,
when $n=35$, and in most other cases Theorem A gives a better upper bound.

These results, together with explicit additional information in \cite{BBH, BU} (for $n\leq8$), Examples~\ref{ex:a9}, \ref{10}, Lemma~\ref{ex:a11}, and Corollary~\ref{gamma S12}, give the exact values of the $\gamma$-function for small $n$, summarised in Table~\ref{tbl:exact}.
{\openup 2pt
\begin{center}
 \begin{table}[H]
\begin{tabular}{l|llllllllll}\hline
$n$&3&4&5&6&7&8&9&10&11&12\\ \hline
$\gamma(S_n)$&2&2&2&2&3&3&4&3&   5&4\\ %\hline
$\gamma(A_n)$&-&2&2&2&2&2&3&3&4&4\\  \hline
\end{tabular}
\caption{Values of $\gamma(S_n)$ and $\gamma(A_n)$ for small $n$.}\label{tbl:exact}
 \end{table}
\end{center}
}

In particular, we observe that
 $\gamma(S_n)$ is not strictly increasing with $n$:
$$
\gamma(S_9)=4>\gamma(S_{10})=3.
$$
Moreover $\gamma(S_n)$ can ``jump":
$$
\gamma(S_{10})=3,\quad \gamma(S_{11})=5.
$$

Many aspects of the behavior of the $\gamma$-functions remain obscure. It would be good to know, for example, the exact value of $\gamma(A_p)$ for primes $p$. We do not know whether or not $\gamma(S_n)$ and $\gamma(A_n)$ are surjective or if there exists $n\in \mathbf N$ such that $\gamma(S_n)<\gamma(A_n).$
%Other bounds can be found also in Proposition \ref{two primes power}.
\medskip

\section{Number theoretic preliminaries}\label{basic number}
Here we present some results from number theory which will be used to prove the asymptotic estimates for $\gamma(S_n)$ and $\gamma(A_n).$
 For these topics
the general reference is the book of H. Shapiro \cite{SH}. We recall the essential definitions and properties for the convenience of the reader. We are indebted to F. Luca for most of the ideas of the proof of Lemma \ref{asymptoticI}.

As usual, we denote by:
\begin{itemize}\item[i)] $(a,b)$ the greatest common divisor of the integers $a$ and $b$;
\item[ii)] $\phi:\mathbf N^*\rightarrow \mathbf N^*$ the Euler function given by
$$
\phi(n)=|\{k\in \mathbf N\quad:\quad (k,n)=1, \quad 1\leq k\leq n\}|
$$
where $\mathbf N$ denotes the set of non-negative integers and $\mathbf N^* = \mathbf N\setminus\{0\}$;
\item[iii)]  $\nu:\mathbf N^*\rightarrow \mathbf N$ the function which gives the number of distinct prime divisors of $n$. We need the following fact:
\begin{equation}\label{nu}
|\{d\in \mathbf N\ :\ d\mid n,\  d\ \mbox{ square free}\}|=2^{\nu(n)}.
\end{equation}

\item[iv)]   $\mu:\mathbf N^*\rightarrow \{-1,0,1\}$ the M\"{o}bius function given by $\mu(1)=1$ and \[
\mu\left( n\right) =\left\{
\begin{array}{lll}
0 &\quad \quad &\mbox { if $n$ is not square free}\\
\\
(-1)^{\nu \left( n\right)}& \quad \quad &\mbox{ if $n$ is square free}
\end{array}%
\right.
\]%
We will use the relation
\begin{equation}\label{mu}
\sum _{d\mid n}\mu(d)=  \left\{
\begin{array}{lll}
1 &\quad \quad &\mbox { if $n=1$ }
\\
0 & \quad \quad &\mbox{ if $n>1$ }
\end{array}%
\right.
\end{equation}
 and also the link between the Euler function and the   M\"{o}bius function given by    \begin{equation}\label{phi}\displaystyle{\sum _{d\mid n}\frac{\mu(d)}{d}=\frac{\phi(n)}{n} }.
 \end{equation}

 \end{itemize}

For $f,\ g$ real functions defined over an upper unbounded domain, we write:
\begin{itemize}%\item[vi)] $f=o(g)$ if
%$\displaystyle{\lim_{x\rightarrow +\infty}\frac{f(x)}{g(x)}}=0;$

\item[v)] $f\sim g$ if $\displaystyle{\lim_{x\rightarrow +\infty}\frac{f(x)}{g(x)}} = 1,$ and in this case we say that ``$f$ is asymptotic to $g$".
\end{itemize}
%Clearly $f\sim g$ is equivalent to $f=g+o(g)$ and to $f-g=o(g)$.

\begin{lemma}\label{asymptotic} For any $\alpha\in[0,1)$, \quad $\displaystyle{\lim_{n\rightarrow +\infty}\frac{2^{\nu(n)}n^{\alpha}}{\phi(n)}}=0.$

\end{lemma}
\begin{proof} In \cite[p. 347]{SH} we find
$$
\nu(n)\leq (c_4+1)\displaystyle{\frac{\log\,n}{\log(\log\,n)}}
$$
where $c_4>0$ is a constant and on \cite[p. 341]{SH} we find
$$
\phi(n)>c\displaystyle{\frac{n}{\log(\log\,n)}}
$$
for all $n\geq 3$, where $c>0$ is a constant. It follows that
$$
\frac{2^{\nu(n)}n^{\alpha}}{\phi(n)}\leq
\frac{\log(\log\,n)}{c}\,n^{\mbox{\small{$\frac{c_4+1}{\log_{_{2}}e}\frac{1}{\log(\log\,n)}+\alpha-1$}}}
$$
and with the substitution $m=\log(\log\,n),$ it is immediately seen that for $\alpha\in[0,1),$
$$
\log(\log\,n)\,n^{\mbox{\small{$\frac{c_1}{\log(\log\,n)}+\alpha-1$}}}
$$
tends to $0$ when $n$ goes to $+\infty,$ for any constant $c_1>0.$\\

\end{proof}

\begin{lemma}\label{sym} If $f_1\sim g_1,\ f_2\sim g_2$ and if the limit $\displaystyle{\lim _{x\rightarrow +\infty}\frac{g_1}{g_2}}$ exists and is not equal to $-1,$ then $f_1+f_2\sim g_1+g_2$.
\end{lemma}

\begin{proof} Set $\ell:= \lim _{x\rightarrow +\infty}\frac{g_1}{g_2}$ and note that $\ell\ne -1$. Now
\[
\frac{f_1+f_2}{g_1+g_2}= \frac{f_1}{g_1}\cdot \frac{1}{1+g_2/g_1} + \frac{f_2}{g_2}\cdot \frac{1}{1+g_1/g_2}
\]
and since by assumption the limits, as $x\rightarrow +\infty$, of  $\frac{f_1}{g_1}, \frac{f_2}{g_2}, \frac{1}{1+g_2/g_1}, \frac{1}{1+g_1/g_2}$ all exist and are finite, it follows that
$\lim_{x\rightarrow +\infty} \frac{f_1+f_2}{g_1+g_2}$ exists and is equal to
\[
({\lim_{x\rightarrow +\infty} \frac{f_1}{g_1}})\cdot \frac{1}{1+\ell^{-1}} + (\lim_{x\rightarrow +\infty} \frac{f_2}{g_2})\cdot \frac{1}{1+\ell} = 1
\]
and hence $f_1+f_2 \sim g_1+g_2$.
\end{proof}

\begin{lemma}\label{asymptoticI} Let $n\in \mathbf{N}^*$ and let $0\leq x<y\leq n$ with $x,y\in \mathbf{R}$. For any interval $I$ with extremes $x$ and $y$, define
$$
\phi(I;n)=|\{ i\in \mathbf{N}^*\ :\ i\in I,\ (i,n)=1  \}|.
$$
Then $$|\phi(I;n)-\displaystyle{\frac{\phi(n)}{n}}(y-x)|\leq 2^{\nu(n)+1}
$$
with $\nu(n)$ as in iii) above. Moreover, if $y-x\sim c\,n^{\beta}$ for some $\beta\in (0,1]$ and $c>0$, then
$
\phi(I;n)\sim \displaystyle{\frac{\phi(n)}{n}}(y-x).
$
\end{lemma}

\begin{proof} The relation \ref{mu} implies that
$$
\phi(I;n)=\sum_{i\in I}  \sum_{d\mid (n,i)}\mu(d)=\sum_{d\mid n}\mu(d)\sum_{\mbox{\tiny{$\begin{array}{c}
i\in I\\
d\mid i
\end{array}$%
}}}1.
$$
Now
$$
\displaystyle{\sum_{\mbox{\tiny{$\begin{array}{c}
i\in I\\
d\mid i
\end{array}$%
}}}1=\lfloor\frac{y-x}{d}\rfloor+\delta_d= \frac{y-x}{d}-\theta_{n,d}+\delta_d,}
$$
for some $\delta_d\in\{-1,1,0\}$ and $0\leq \theta_{n,d}<1.$
Thus, by the equality \ref{nu}, we get
\begin{eqnarray*}
\phi(I;n)&=&\sum_{d\mid n}\mu(d)\frac{y-x}{d}+\sum_{d\mid n}\mu(d)(\delta_d-\theta_{n,d})\\
&=&\frac{\phi(n)}{n}(y-x)+\sum_{\mbox{\tiny{$\begin{array}{c}
d\mid n\\
d\  \mbox{square free}
\end{array}$%
}}}\mu(d)(\delta_d-\theta_{n,d}).
\end{eqnarray*}
It follows, by the equation \ref{phi}, that
$$
|\phi(I;n)-\frac{\phi(n)}{n}(y-x)|\leq \sum_{\mbox{\tiny{$\begin{array}{c}
d\mid n\\
d\  \mbox{square free}
\end{array}$%
}}}|\delta_d|+|\theta_{n,d}|<\sum_{\mbox{\tiny{$\begin{array}{c}
d\mid n\\
d\  \mbox{square free}
\end{array}$%
}}}2=2^{\nu(n)+1}.
$$
If $y-x\sim c\,n^{\beta}$ for some $\beta\in (0,1]$ and $c>0$, then
$$
\left|\frac{\phi(I;n)}{\frac{\phi(n)(y-x)}{n}}-1\right|\leq
\frac{2^{\nu(n)+1}n}{\phi(n)(y-x)}= \frac{2^{\nu(n)+1}\,n^{1-\beta}}{\phi(n)}\cdot\frac{n^{\beta}}{y-x}.
$$
By Lemma \ref{asymptotic},
$\displaystyle{\lim_{n\rightarrow +\infty}\frac{2^{\nu(n)+1}\,n^{1-\beta}}{\phi(n)}}=0$ and, by assumption, $\displaystyle{\lim_{n\rightarrow +\infty}\frac{n^{\beta}}{y-x}=1/c},$ so that
$\displaystyle{\lim_{n\rightarrow +\infty}\frac{2^{\nu(n)+1}\,n^{1-\beta}}{\phi(n)}\frac{n^{\beta}}{y-x}}=0$ and therefore
$
\phi(I;n)\sim \displaystyle{\frac{\phi(n)}{n}}(y-x).
$\end{proof}

 We close this section with some arithmetical lemmas.

 \begin{lemma}\label{min-prime-odd} Let $n\in\mathbb{N}$ with $n$ odd and $n\geq3$, and let $p$ be its smallest prime divisor.
 \begin{description}\item[(a)] If $n$ is not a prime, the square of a prime or the product of twin primes, then $ p < \sqrt{n} -1;$
 \item[(b)] If $n$ is a prime, the square of a prime or the product of twin primes, then $ p>  \sqrt{n} -1.$
 \end{description}
 \end{lemma}

\begin{proof} Assertion $(b)$ is obvious. To prove $(a),$ let $n$ be odd and $p$ its smallest prime divisor, so $p\geq3$. Assume that $n\ne p, p^2$ and if $p+2$ is prime assume also that $n\ne p(p+2)$. Let $m:=n/p$ so $m>1$ and $m$ is odd. Then $m>p$ by the minimality of $p$ and since $n\ne p^2$. If $m=p+2$ then by assumption $m$ is not prime, but then $m\geq q^2$ where $q$ is the least prime dividing $m$, and $q\geq p$. Thus $n\geq p^3$ and so $p\leq n^{1/3} < n^{1/2}-1$. Thus we may assume that $m>p+2$, and hence $m\geq p+4$ since $m$ is odd. It follows that $\sqrt{n}\geq\sqrt{p(p+4)}>p+1.$
\end{proof}

The next lemma is an adaptation of Theorem 3.7A.2 in \cite{SH}, which makes also use of the
following result.

\begin{theorem}\label{30}{\em (3.7A.1 in \cite{SH})} If $n\in \mathbb{N}$ and $n>30,$ then there exists $m\in \mathbb{N}$, not a prime, such that $1<m<n$ and $(m,n)=1.$
\end{theorem}

For each $n\in \mathbb{N^*}$ we set
\begin{equation}\label{p0}
P(n):=\{x\in \mathbb{N^*}\ : \ x\  \hbox{is a prime and}\ x\nmid n\},\quad\mbox{and}\quad p^0(n):= \min\, P(n).
\end{equation}
\begin{lemma}\label{min-prime-even} Let  $n\in \mathbb{N}$ with $4\mid n$ and let $p^0=p^0(n)$ as in {\rm (\ref{p0})}. If  $n\geq 16,\ n\neq 24,$ then $ p^0\leq \sqrt{n} -1.$ %Moreover if $n>60,$ then $p^0\leq \sqrt{n-4} -1.$
\end{lemma}
\begin{proof} Let  $n\in \mathbb{N}$ with $4\mid n.$ Then $p^0\geq 3.$ Assume first $n>60;$ then $n/2>30$ and by Theorem \ref{30}, there exists a minimal $s\in \mathbb{N}$ with $1<s<n/2$, $(s,n/2)=1$ and $s$ not a prime. Since any proper divisor of $s$ is coprime with $n/2$ and lies in the same interval $(1,n/2)$, the minimality of $s$ implies that any proper divisor of $s$ is a prime, that is, $s$ is a product of two primes. If $s=pq$ with $q>p$ then $s>p^2$. Thus $1<p^2<s<n/2$ and, since $p\mid s,$ we get also $(p^2,n/2)=1,$ contradicting the minimality of  $s.$ It follows that
$s=p^2$ for some prime $p$. Since $4$ divides $n,$ the condition $(s,n/2)=1$ is equivalent to $(s,n)=1$. Hence $p\nmid n$ and thus  $p\geq p^0.$ So we have  ${\frac{n}{2}}>s=p^2\geq (p^{0})^2$, and hence $p^0\leq\sqrt{n/2}\leq \sqrt{n} -1.$ %$\sqrt{n-4}-1$.

The cases $16\leq n\leq 60$ are easily checked by hand and, with the single exception of $n=24,$ we find that the inequality $ p^0\leq \sqrt{n} -1$ holds.
\end{proof}

Note that the inequality $ p^0\leq \sqrt{n} -1$ is sharp, for example, for $n=36$.\\
 For each $n\in \mathbb{N^*}$  define
\begin{equation}\label{a}
X(n):=\{x\in \mathbb{N^*}\ : \ x\nmid n\},\quad\mbox{and}\quad a(n):= \min \, X(n).
\end{equation}
  Observe that in general $a(n)\neq p^0(n):$ for example, $a(6)=4<p^0(6)=5.$ Note also that $a(n)=2$ if and only if $n$ is odd and $a(n)=3$ if and only if $n$ is even and not divisible by $3$.
  We summarise various properties of $a(n)$ in the following lemma.
 \begin{lemma}\label{a-properties} For any $n\in \mathbb{N^*}:$ \begin{description}\item[(a)] $a(n)\leq p^0(n);$
 \item[(b)] $a(n)= p^0(n)$ if and only if $a(n)$ is a prime;
 \item[(c)] $a(n)=p^{\alpha}$ for some prime $p$ and $\alpha\in \mathbb{N},$ such that $(n,a(n))=p^{\alpha-1}.$
 \end{description}
\end{lemma}
\begin{proof} (a), (b) Clearly $X(n)$ contains the set $P(n)$ defined in (\ref{p0}), and thus $a(n)\leq p^0(n).$ If $a(n)$ is prime, then $a(n)\in P(n),$ which gives $a(n)\geq p^0(n)$ and thus $a(n)= p^0(n).$

(c) Suppose that $a(n)$ is not a prime power. Then, for each prime $p$ dividing $a(n)$, the $p$-part $p^k$ of $a(n)$ is less than $a(n)$, and hence does not lie in $X(n)$. Thus $p^k|n$, and since this holds for all primes $p$ dividing $a(n)$, it follows that $a(n)$ divides $n,$ which is a contradiction.

Thus $a(n)=p^{\alpha}$ for some $p$ prime and $\alpha\in \mathbb{N}.$ Since $p^{\alpha-1}<a(n)$, we have that $p^{\alpha-1}$ divides $n$ and hence divides $(n,a(n)).$ On the other hand, $(n,a(n))$ divides $a(n)=p^{\alpha}$, and since $a(n)\nmid n$, it follows that $(n,a(n))=p^{\alpha-1}$.
 \end{proof}

The eventuality $a=a(n)$ prime is a very common feature:  for $n\leq 100$,  $a(n)$ is prime unless $n\in\{6,18, 30, 54, 66, 78, 90\}$, in which cases $a(n)=4.$ On the other hand it is easy to construct examples of natural numbers $n$ for which $a(n)$ is any given prime power. For instance $a(3\cdot5\cdot 7\cdot 8)=9.$

\section{
Group theoretic preliminaries}\label{basic permutations}

\subsection{Conjugacy and cycle type}\label{sect:cyctype}

If $\sigma\in S_n$ decomposes as a product of $r$ disjoint cycles of lengths $\ell_1,\dots,\ell_r,$ with each $\ell_i\geq 1$, we say that $\sigma$ is of {\em type} $L=[\ell_1,\dots,\ell_r]$. Recall that two permutations in $S_n$ are conjugate if and only if they are of the same type. Thus

\begin{tabular}{l}
\emph{$\delta=\{H_i<S_n\  :\   i=1 \dots k \}$ is a basic set for $S_n$ if and only if the $H_i$ are}\\ \emph{pairwise non-conjugate proper subgroups of $S_n$ and, for each type $L$ for}\\ \emph { permutations in $S_n$, some $H_i$ contains a permutation of type $L$.
}
\end{tabular}\\

For $\sigma\in A_n $ of type $L=[\ell_1,\dots,\ell_r]$,  the
$S_n$-conjugacy class $\sigma^{S_n}$ of $\sigma$ is either an $A_n$-conjugacy class or a union of two $A_n$-classes, say $\sigma^{S_n}=\sigma_1^{A_n}\cup \sigma_2^{A_n}.$ In the latter case we say that $\sigma$, its class $\sigma^{S_n}$, and its type $L$ are {\em split}. This happens if and only if $C_{S_n}(\sigma)\leq A_n$, which in turn holds if and only if the $\ell_i$ are  pairwise distinct odd positive integers. If the $S_n$-conjugacy class of $\sigma \in A_n$ does not split we can trivially write again $\sigma^{S_n}=\sigma_1^{A_n}\cup \sigma_2^{A_n}$ with $\sigma_1= \sigma_2=\sigma.$
Now

\begin{tabular}{l}
\emph{ $\delta=\{H_i<A_n\  :\   i=1 \dots k \}$ is a basic set for $A_n$ if and only if the $H_i$ are}\\ \emph{pairwise non-conjugate proper subgroups of $A_n$ and, for each type $L$ for }\\ \emph {permutations in $A_n$, a representative of each $A_n$-conjugacy class of type $L$}\\ \emph {lies in some $H_i$ (possibly different $H_i$ for different classes).}
\end{tabular}\\

The next lemma gives sufficient conditions to guarantee that a subgroup $H$ which meets a split $\sigma^{S_n}$ contains elements from both $A_n$-classes. For a subgroup $H$ of $S_n$ let $H^{A_n}=\{H^x\ :\ x\in A_n\}$ and $H^{S_n}= \{H^x\ :\ x\in S_n\}$.

\begin{lemma}\label{split} Let $H< S_n.$ Then
\begin{description}
\item[(a)] $H^{A_n}=H^{S_n}$ if and only if $N_{S_n}(H)\not \leq A_n$
   \item[(b)] If  $K$ is a maximal subgroup of $S_n$, $K\neq A_n$,  and if $H=K\cap A_n,$ then $N_{S_n}(H)\not \leq A_n$.
  \item[(c)] If $H\leq A_n$ contains a permutation of split type $L$ and $N_{S_n}(H)\not \leq A_n$, then $H$ contains elements from both $A_n$-classes corresponding to $L.$

   \end{description}
\end{lemma}

\begin{proof}(a) Under the conjugation action of $S_n$ on the set $H^{S_n}$,  $H^{A_n}$ is the orbit containing $H$ of the subgroup $A_n$; and the stabilizer of $H$ in $A_n$ is $N_{A_n}(H).$ Thus $|H^{A_n}|=|A_n:N_{A_n}(H)|$ while $|S_n:N_{S_n}(H)|=|H^{S_n}|.$ If $N_{S_n}(H)\not \leq A_n$ we have $A_n N_{S_n}(H)>A_n$ and hence $S_n=A_n N_{S_n}(H)$. In this case $|A_n:N_{A_n}(H)|=|A_n:A_n\cap N_{S_n}(H)|=|S_n:N_{S_n}(H)|.$ On the other hand, if $N_{S_n}(H) \leq A_n$, then $N_{S_n}(H) =N_{A_n}(H) $, and $|S_n:N_{S_n}(H)|=2|A_n:N_{A_n}(H)|\neq |A_n:N_{A_n}(H)|.$ \medskip

(b) Let $A_n\neq K <\cdot\, S_n$ and $H=K\cap A_n.$ By maximality $K\not\leq A_n.$ Since $H$ is normal in $K,$ we have $N_{S_n}(H)\geq K$ and thus $N_{S_n}(H)\not\leq A_n.$\medskip

(c) Let $H\leq A_n$ with $N_{S_n}(H)\not \leq A_n$, so by part (a), $H^{A_n}=H^{S_n}$. Suppose that $\sigma\in H$ is of split type and let $\sigma^{S_n}=\sigma_1^{A_n}\cup \sigma_2^{A_n}$ with $\sigma_1,\ \sigma_2$ not conjugate in $A_n$. We may assume that $\sigma=\sigma_1\in H$. If $x\in S_n$ is such that $\sigma_1^x=\sigma_2$ then, since $H^{A_n}=H^{S_n}$, we have $\sigma_2\in H^x=H^a$ for some $a\in A_n$, and hence $\sigma_2^{a^{-1}}\in H$. Thus $H$ intersects both
$\sigma_1^{A_n}$ and $\sigma_2^{A_n}.$
\end{proof}

We will write $[\ell_1,\dots,\ell_r]\in H,$ for some $H\leq S_n$, to signify that some permutation of type $[\ell_1,\dots,\ell_r]$ belongs to the subgroup $H.$

Our strategy for obtaining lower bounds for the numbers $\gamma (G)$ with $G=S_n,\ A_n$ is usually the following: we consider a family $\mathcal{F}$ of types of permutations  with $|\mathcal{F}|=m$, and identify all the maximal subgroups containing a permutation of one of these types. If each of the maximal subgroups that arise contains at most $k$ types of permutations in $\mathcal{F}$, then $\gamma (G)\geq m/k.$

\subsection{Primitive and imprimitive subgroups}

A transitive $H\leq S_n$ is primitive if the only $H$-invariant partitions of $\Omega=\{1,\dots,n\}$ are the trivial ones with parts of size 1, or with one part; $H$ is imprimitive if it is not primitive. The parts of an $H$-invariant partition are called blocks for $H$. Most permutation group concepts we use are standard and can be found in any of \cite{CA,DM,WI}.

When dealing with transitive maximal subgroups, we divide our analysis into the imprimitive and primitive cases. We state here some results which can be applied to get control of both kinds of transitive subgroups.  For $\sigma\in S_n$ we denote by $supp(\sigma)$ the set of points moved by $\sigma$, called the \emph{support} of $\sigma$. Proof of the first result is straightforward and is omitted.

\begin{lemma}\label{imprimitive}Let $G\leq S_n$ be an imprimitive group with set $\mathcal{B}$ of
blocks,  and let $\sigma\in G$ be the product of disjoint
cycles $\gamma_1,\dots,\gamma_r.$ Let
$$
\mathcal{B}[\gamma_i]=\{B\in\mathcal{B}\ :\ B\cap supp(\gamma_i)\neq \varnothing\},
$$
the set of
blocks that meet the support of $\gamma_i$. Then:
\begin{description}\item [(a)] $\mathcal{B}[\gamma_i]\cap
\mathcal{B}[\gamma_j]=\varnothing$ or $\mathcal{B}[\gamma_i]=\mathcal{B}[\gamma_j];$
\item [(b)] if $B_1,\ B_2\in \mathcal{B}[\gamma_i]$ then  $|B_1\cap supp(\gamma_i)|=|B_2\cap supp(\gamma_i)|.$
\end{description}\end{lemma}

If $\Delta\subseteq\Omega$ and $H\leq S_n,$ we denote by $H_{(\Delta)}$ the pointwise stabilizer of $\Delta$ and by $H_{ \Delta}$ the setwise stabilizer of $\Delta.$  We will apply several times a well known result on the transitive action of a stabilizer in a primitive group, and a classic old result by Marggraf.

\begin{lemma}\label{WI 13.5} {\em\cite[Theorem 13.5]{WI}} Let $H\leq S_n$ be a primitive group and let $H_{(\Delta)}$ be transitive on $\Omega-\Delta=\Gamma. $ If $2\leq |\Gamma|<n/2,$ then $H=S_n$ or $H=A_n.$
\end{lemma}
\begin{lemma}\label{WI 13.8}{\em \cite[Theorem 13.8]{WI}} A primitive group of degree $n$, which contains a cycle of degree $m$ with $1<m<n,$ is $(n-m+1)$-fold transitive.
\end{lemma}

To deal with $n$-cycles we need a deeper result, relying on the classification of finite simple groups.
Collecting results by  Burnside, Galois, Ritt,
Schur and Feit (as done in \cite{J}) together with \cite[Corollary 2]{J}, it can be stated as follows.

\begin{lemma} Let $H\leq S_n$ be a primitive group containing an $n$-cycle.
\begin{description}\item[(a)] If $H$ is simply transitive or solvable, then $H\leq AGL_1(p)$ with $n=p$ a prime, or $H=S_4$ with $n=4.$
\item[(b)] If $H$ is nonsolvable
and doubly transitive, then one of the following holds.
\begin{description}\item[(i)]
 $H=S_n$ for some $n\geq5$, or $H=A_n$ for some odd $n\geq5$;
\item[(ii)] $PGL_d (q)\leq H\leq P\Gamma L_d(q)$, acting on $n=(q^d -1)/(q -1)$ points or hyperplanes;

\item[(iii)] $H= PSL_2(11),\  M_{11}$ or $M_{23}$ with $n= 11, 11$ or $23$ respectively.
\end{description}
\end{description}
\label{Feit}
\end{lemma}

If $G\leq S_n$ and $g\in G$, then the \emph{degree} of $g$
on $\Omega$  is the number of points moved by $g$, that is to say, the size of $supp(g)$. The \emph{minimum
degree} of $G$, denoted by $\mu(G)$, is the minimum degree of a
nontrivial element in $G.$ Results on minimum degrees can be useful
for identifying primitive permutation groups.

\begin{lemma}
 \label{min deg}{\em\cite[Theorem 2 and Corollary 3]{LS}}
Let $G\leq S_n$ be primitive, with $G\not\geq A_n.$ Then $\mu(G)> 2(\sqrt{n}-1).$ Moreover if $G$ is $2$-transitive then $\mu(G)\geq n/3$.
%. or $n=\binom{m}{k}^r$ and $A_m^r\leq G\leq S_m\wr S_r$ in product action, for some $m\geq5$, and $r,k\geq 1$ with $rk>1$, where the action of $S_m$ is on $k$-element subsets of $\{1,\dots,m\}$.
\end{lemma}

Now we state a useful lemma which ensures control of the primitive subgroups containing certain types of permutations. When we write $[m,\dots,m,k]$ as a type for permutations in $S_n$, for distinct positive integers $m,k$, we mean that $m \mid (n-k)$ and  permutations with this type have $\frac{n-k}{m}$ cycles of length $m$ and one cycle of length $k$.

\begin{lemma}
\label{p-permutations}\begin{description}
\item[(a)]

Let $n\in \mathbb{N},\ n$ where $n$ is odd, not a prime, the square of a prime or the product of twin primes. If $p$ is the smallest prime divisor of $n$, then  no primitive proper subgroup of $S_n$ contains a permutation of type $[p,\dots,p, 2p]$.
\item[(b)] Let $n\geq 10$ be even with $4\nmid n.$ Then no primitive proper subgroup of $A_n$ contains
 a permutation of type $[2,\dots,2, 4]$.
\item[(c)] Let $n\geq 12$ be even with $4\mid n.$ Then no primitive proper subgroup of $A_n$ contains a permutation of type $[2,\dots,2, 6]$.

\end{description}
\end{lemma}

\begin{proof}(a) Let $n$ be odd and not a prime, the square of a prime or the product of twin primes. Let $p$ be the smallest prime divisor of $n$, so $p\geq 3$. Note that necessarily $n\geq 21.$ Let $H<S_n$ be primitive and $\sigma\in H$ a permutation of type $[p,\dots,p, 2p]$. Since $\sigma$ is an odd permutation, $H\not\geq A_n,$ as otherwise $H=S_n.$ Now $H$ contains also $\sigma^p,$ a permutation of type $[1,\dots,1,2,\dots,2]$ which moves $2p$ points. Thus $\mu(H)\leq 2p$. Since by Theorem \ref{min deg},  $\mu(H)>2(\sqrt{n}-1)$, it follows that $p>\sqrt{n}-1$  contradicting Lemma \ref{min-prime-odd}.

\medskip\noindent
(b) Let $n\geq 10$ be even with $4\nmid n$ and $H<A_n$ primitive with $[2,\dots,2, 4]\in H.$ Then $H$ contains also a permutation of type $[1,\dots,1,2,2]$ and, as in part (a) we get
$4>2(\sqrt{n}-1),$ which gives $n<9,$ a contradiction.

\medskip\noindent
(c) If $n=12$ the only primitive proper subgroup of $A_{12}$ is $M_{12}$ which does not contain permutations of type $[2,2,2,6]$ (see \cite{WAT}).
So suppose that $n\geq 16$ is even with $4\mid n$, and that $H<A_n$ is primitive with $[2,\dots,2, 6]\in H.$
Then $H$ contains also a permutation of type $[1,\dots,1,3,3]$ and, arguing as before, we get
$6>2(\sqrt{n}-1),$ which gives $n<16,$ contrary to our assumption.  \end{proof}

\subsection{Projective linear groups}\label{sub:singer}

In order to get new exact values for the gamma function when the degree is the product of at most two prime numbers, we need to characterise the projective linear groups containing $(n-1)$-cycles.

\begin{lemma}\label{nminusone}
Let $H$ be a group such that $PSL_d(q)\leq H\leq P\Gamma L_d(q),$ with $d\geq 2$, acting on the set $\Omega$ of $n=(q^d-1)/(q-1)$ points or hyperplanes. Then $H$ contains an $(n-1)$-cycle if and only if $d=2$ and either $q$ is a prime or $(q,H)=(4,P\Gamma L_2(4))$.
\end{lemma}

\begin{proof}
If $d=2$ and $q$ is a prime, then $n=q+1$ and $PSL_2(q)$, and hence also $H$, contains a $q$-cycle. Also if $H=P\Gamma L_2(4)\cong S_5$, then $G$ contains a 4-cycle. Suppose conversely that $H$ is as in the statement and that $H$ contains an $(n-1)$-cycle $g$. Let $p$ be the prime dividing $q$, and let $\alpha\in\Omega$ be the point fixed by $g$. Suppose first that $d=2$. Then $H_\alpha$ is a soluble primitive permutation group on $\Omega\setminus\{\alpha\}$ of degree $n-1$ containing an $(n-1)$-cycle. Hence by Lemma~\ref{Feit}, either $q$ is prime or $q=4$ and $H_\alpha=S_4$. In the latter case, $H=P\Gamma L_2(4)$.

Thus we may assume that $d\geq3$. Then the set of lines containing $\alpha$ (with $\alpha$ removed from each) forms a set $\Sigma$ of blocks of imprimitivity for $H_\alpha$ in $\Omega\setminus\{\alpha\}$, each of size $q$.
The permutation group $L=H_\alpha^\Sigma$ induced by $H_\alpha$ on $\Sigma$ satisfies $PGL_{d-1}(q)\leq L\leq P\Gamma L_{d-1}(q)$ and $L$ contains an $\frac{n-1}{q}$-cycle $g^\Sigma$ (induced by $g$). By \cite[Corollary 2]{J}, either $\la g^\Sigma\ra$ is a Singer subgroup of $PGL_{d-1}(q)$, or $(d-1,q)=(2,8)$.

Now $\la g\ra$ induces a transitive abelian subrgoup $\la g^\Sigma\ra$ on $\Sigma$. Hence $\la g^\Sigma\ra$ is regular, and so $g^{(n-1)/q}$ fixes each of the lines through $\alpha$ setwise, and has order $q$. The subgroup of $P\Gamma L_d(q)_\alpha$ fixing all these lines setwise is an extension of an elementary abelian group $N$ of order $q^{d-1}$ by a cyclic group of order $q-1$. Since  $g^{(n-1)/q}$ is a $p$-element it follows that $g^{(n-1)/q}\in N$, and since $g^{(n-1)/q}$ has order $q$ while $N$ has exponent $p$, it follows that $q=p$ is prime. Thus $\la g^\Sigma\ra$ is a Singer subgroup of $PGL_{d-1}(q)$.  It follows that $\la g\ra$, acting by conjugation on $N$, is fixed point free. However  this contradicts the fact that $g$ centralises the non-identity element  $g^{(n-1)/q}$ of $N$.
\end{proof}

\section{
Basic facts on normal coverings}\label{normal general}
In this brief section we establish a link between $\gamma(S_n)$  and $\gamma(A_n)$ and derive upper bounds for $\gamma(S_n)$ and $\gamma(A_n)$.% using the approach of \cite{MA}.

\begin{lemma}
\label{SA} Any normal covering of $S_n$ with maximal components all different from $A_n$ defines, by intersection, a normal covering of $A_n.$ In particular if $\gamma(S_n)$ is realized by a normal covering in which all basic components are maximal subgroups and none equals $A_n$, then
$\gamma(A_n)\leq \gamma(S_n).$
\end{lemma}

\begin{proof}  Let $\delta=\{H_1,\dots,H_k\}$ be a basic set for $S_n$  with $H_i$ maximal in $S_n$ and $H_i\neq A_n$ for any $i=1,\dots, k.$ Then, by Lemma \ref{split}(a) and (b),
$$
A_n=A_n\cap [\,\bigcup_{g\in S_n} \bigcup_{i=1}^k H_i^g\,]=\bigcup_{g\in S_n} \bigcup_{i=1}^k [A_n\cap H_i]^g=\bigcup_{g\in A_n} \bigcup_{i=1}^k [A_n\cap H_i]^g
$$
which means that $\delta_{A_n}=\{A_n\cap H_1,\dots,A_n\cap H_k\}$ generates a normal covering for $A_n.$\end{proof}

Sometimes we use the previous Lemma with no explicit reference. Here we apply it and Lemma~\ref{split} to obtain upper bounds for $\gamma(A_n)$ for any $n$ and for $\gamma(S_n)$ in the case when $n$ is not prime. The symmetric case for $n$ prime is completely determined in Proposition~\ref{sym prime}.

\begin{proposition} \label{upper} \begin{description}
\item[(a)] Let $n\in \mathbb{N}$ be composite, and let $p$ be the least prime divisor of $n$.
Then
{\openup 3pt
\[
      \gamma(S_n)\leq \left\{\begin{array}{ll}
                              1+\frac{n}{2}\left(1-\frac{1}{p}\right)\leq \frac{n-1}{2}&\mbox{if $n$ is odd}\\
                              & \\

\left\lfloor\frac{n+4}{4}\right\rfloor &\mbox{if $n$ is even}\\
                             \end{array}\right.
     \]
}\noindent
and the set
\[
\delta(A)=\{ A,\ S_k\times S_{n-k}\ :\ 1\leq k<n/2,\ p\nmid k\}
\]
for $A:= S_p\wr S_{n/p}$ and for $A:=S_{n/p}\wr S_p,$ is a basic set for $S_n.$
\item[(b)] Let $n\geq 4.$ Then

{\openup 3pt
\[
      \gamma(A_n)\leq \left\{\begin{array}{ll}
\left\lfloor\frac{n+3}{3}\right\rfloor &\mbox{if $n$ is odd}\\
  & \\
\left\lfloor\frac{n+4}{4}\right\rfloor &\mbox{if $n$ is even.}\\
                             \end{array}\right.
\]
}\noindent
If $n$ is even, the set $\delta(A)$ defines a basic set for $A_n$ by intersection, for each $A$, while if $n$ is odd then, for any maximal subgroup $K$ of $S_n$ containing an $n$-cycle and not equal to $A_n$ (such $K$ exist),  the set
\[
\widehat{\delta}(K)=\{ K\cap A_n,\ [S_k\times S_{n-k}]\cap A_n\ :\ 1\leq k\leq \left\lfloor n/3\right\rfloor\}
\]
 is a basic set for $A_n.$
\end{description}
\end{proposition}

\begin{proof}
(a) Both $A:= S_p\wr S_{n/p}$ and $A':=S_{n/p}\wr S_p$ are maximal subgroups of $S_n.$ For $1\leq k<n/2$ with $k$ coprime to $p$, let $B_k:=S_k\times S_{n-k}$. To prove that both $\delta(A)$ and $\delta(A')$ are basic sets for $S_n$, we note first that each $n$-cycle lies in some conjugate of $A$ and in some conjugate of $A'$. Next, the type $[k,n-k]\in B_k$ if $p\nmid k$, and $[k,n-k]\in A, A'$ otherwise. Now consider a type $L=[\ell_1,\dots,\ell_r]$ with $r\geq3$ and $\sum_{i=1}^r\ell_i=n$. If each $\ell_i$ is divisible by $p$ then $L\in A$ and $L\in A'$. So assume that there exists $\ell_j$ with $p\nmid \ell_j$. If $\ell_j<n/2$ then $L\in B_{\ell_j}.$ If $\ell_j\geq n/2$ then, since $r\geq3$  all the other entries are less than $n/2;$ since $p\mid n,$ there exists one of them $\ell_i$ such that $p\nmid \ell_i$ and thus $L\in B_{\ell_i}$. Thus $\gamma(S_n)\leq |\delta(A)|=|\delta(A')|$.

The set $U:=\{ k\ :\ 1\leq k<n/2,\ p\nmid k\}$ is in one-to-one correspondence with $U':=\{ k\ :\ n/2<k\leq n,\ p\nmid k\}$, and the sets $U$ and $U'$ are disjoint. The set of all integers less than $n$ and coprime to $p$ is equal to $U\cup U'$ unless $p=2$ and $4 \nmid n$, in which case it is $U\cup U'\cup\{n/2\}$. Thus  if either $n$ is odd or $n$ is divisible by $4$, we have
\[
|\delta(A)|=1+|U| = 1 + \frac{n}{2}\left(1-\frac{1}{p}\right)
\]
which is $\frac{n+4}{4}$ if $n$ is divisible by $4$ (since then $p=2$), and at most $\frac{n-1}{2}$ if $n$ is odd (since then $n/p\geq3$). On the other hand if $n\equiv 2\pmod{4}$, then
\[
|\delta(A)|=1+|U| = 1 + \frac{n-2}{4} = \left\lfloor\frac{n+4}{4}\right\rfloor.
\]

(b) By Lemma~\ref{SA}, if $n$ is composite, then each of $\delta(A)$ and $\delta (A')$ defines by intersection a normal covering of $A_n$ and hence also $\gamma(A_n)\leq |\delta(A)|$. All assertions are now proved if $n$ is even, so assume that $n$ is odd and $n\geq5$. Then $A_n$ does not contain permutations of type $[k,n-k]$ for any $k$. Moreover the $n$-cycles of $A_n$ split into two $A_n$-conjugacy classes and we claim that all $n$-cycles can be covered by a single basic component. There exists a maximal subgroup $K$ of $S_n$ such that $K\neq A_n$ and $K$ contains an $n$-cycle: if $n$ is prime, take $K=AGL_1(n)$ and if
$n$ is composite, take $K=S_q\wr S_{n/q},$ for any prime divisor $q$ of $n.$ Then, by Lemma \ref{split} (b) and (c), $H=K\cap A_n<A_n$ contains elements from both the $A_n$-classes of $n$-cycles. This proves the claim.
All other types occurring in $A_n$ have at least three entries and hence are covered by some subgroup $[S_k\times S_{n-k}]\cap A_n$ where $k$ is an integer with $1\leq k\leq n/3$. Thus $\gamma(A_n)\leq \left\lfloor\frac{n+3}{3}\right\rfloor.$ We note that in general, if $n$ is composite, then this upper bound is better than the upper bound $|\delta(A)|$,  since for $p$ odd we have $1 + \frac{n}{2}\left(1-\frac{1}{p}\right)\geq \frac{n+3}{3},$ with equality if and only if the smallest prime dividing $n$ is $p=3.$
\end{proof}

\section{
2-Cycle decompositions and intransitive components.}\label{intrans}

In this section we consider permutations $\sigma_k\in S_n,\ n\geq 5$ of type
$[k,n-k]$ with $k$ belonging to

\begin{equation}\label{Un}
\mathcal {U}_n=\{k\in \mathbb N\quad:\quad (k,n)=1, \quad 2\leq k<n/2\}.
\end{equation}

\begin{lemma}\label{lem:Un}
The set of types
$$
U_n=\{[k,n-k]\quad:\quad k\in \mathcal {U}_n\}
$$
is in one-to-one correspondence with $\mathcal{U}_n$, and $|\mathcal{U}_n|=\frac{\phi(n)}{2}-1$.
\end{lemma}

\begin{proof}
If distinct $k_1, k_2\in \mathcal {U}_n$ correspond to equal types $[k_1,n-k_1]$, $[k_2,n-k_2]$, then $k_1=n-k_2, $ which cannot happen since $ k_1<n/2$ while
$n-k_2>n/2.$ To find $|\mathcal {U}_n|,$ we consider the sets
 $$\begin{array}{lll}
\mathcal {U}_n^1&=&\{k\in \mathbb N\ :\ (k,n)=1, \  1\leq k<n/2\}\\
\mathcal {U}_n^2&=&\{k\in \mathbb N\ :\  (k,n)=1, \  n/2 <k\leq n-1\}.
   \end{array}
$$
   Since $(n-x,n)=(x,n),$ the map $f:\mathcal {U}_n^1\rightarrow \mathcal {U}_n^2$ defined
   by $f(x)=n-x$ is a bijection between $\mathcal {U}_n^1$ and $\mathcal {U}_n^2.$
    Moreover $\mathcal {U}_n^1\cap \mathcal {U}_n^2=\varnothing$ and
    $$\mathcal {U}_n^1\cup \mathcal {U}_n^2=\{k\in \mathbb N\ :\  (k,n)=1, \
     1\leq k\leq n\}.$$ Thus $|\mathcal {U}_n^1|= |\mathcal {U}_n^2|=
     \displaystyle{\frac{\phi(n)}{2}}$ and $|\mathcal {U}_n|=|\mathcal {U}_n^1|-1=
      \displaystyle{\frac{\phi(n)}{2}}-1.$\\
\end{proof}

The $\sigma_k$ typically lie in $S_k\times S_{n-k}$ and we will show that, with one exception in the even case, $S_k\times S_{n-k}$ is the only maximal subgroup which can contain $\sigma_k$.
\begin{lemma} Let $\sigma_k\in S_n, n\geq 5$ be a permutation of type $[k,n-k]\in U_n$  and $H$ a maximal subgroup of $G=S_n$ or $A_n$ containing $\sigma_k.$ Then we have the following:

\begin{description}
\item[(a)] if $n$ is odd  then $G=S_n$ and $H=S_k\times S_{n-k};$
\item[(b)] if $n$ is even and $G=S_n,$ then $H=S_k\times S_{n-k}$ or $H=A_n;$
\item[(c)] if $n$ is even and $G=A_n,$ then $H=[S_k\times S_{n-k}]\cap A_n.$
\end{description}
In particular any basic set for $G=S_n,$ with $n$ odd, contains
properly
$$
\{S_k\times S_{n-k}\quad:\quad k\in\mathcal{U}_n\}
$$
 and any basic set for $G=A_n$ with $n$ even, contains properly
$$
\{[S_k\times S_{n-k}]\cap A_n\quad:\quad k\in\mathcal{U}_n\}.
$$
\label{two cycles}
\end{lemma}
\begin{proof} If $n$ is odd then $\sigma_k$ is an odd permutation and hence $G=S_n$ in this case. If $H$ is intransitive then, up to conjugacy, $H=[S_k\times S_{n-k}]\cap G$.

Assume now that $H$ is transitive. Suppose first that $H$ is imprimitive preserving a partition with $m\geq 2$ blocks of size $b\geq 2$, where $n=b\,m.$ By Lemma \ref {imprimitive}, each of the two cycles of $\sigma_k$ meets either the same set of blocks or two disjoint sets of blocks. Hence $m\mid (k, n-k)=1$ or $b\mid (k, n-k)=1$ respectively, which is a contradiction.
Thus $H$ is primitive on $\Omega=\{1,\dots,n\},$ and  we may assume that
$$
\sigma_k=(1,\dots, k)(k+1,\dots,n)\in H.
$$
Consider the partition $\Omega=\Gamma\, \cup\, \Delta,$ where $\Gamma=\{1,\dots, k\}$ and $\Delta=\{k+1,\dots,n\}.$ Since $(k,n)=1,$ we have that $\sigma_k^{n-k}$ is a $k$-cycle in $H$. This implies that the pointwise stabilizer $H_{(\Delta)}$ of $\Delta$ in $H$ is transitive on $\Gamma,$ with $1<|\Gamma|=k<n/2= |\Omega|/2$. Thus Lemma \ref{WI 13.5} applies giving, in the case $n$ even, $H=A_n$  and $G=S_n$ as in (b), and in case $n$ odd a contradiction. Parts (a) - (c) are now proved.

Since the only intransitive maximal subgroup of $G$ containing $\sigma_k$ is, up to conjugacy, $H=[S_k\times S_{n-k}]\cap G$, it follows that all the components $[S_k\times S_{n-k}]\cap G$ with $k\in \mathcal{U}_n$ are mandatory for $G=S_n$ if $n$ is odd and  for $G=A_n$ if $n$ is even. Since none of these groups contains an $n$-cycle or an $(n-1)$-cycle, there must be an additional component in either of these cases.
\end{proof}

\begin{proposition} \label{gamma twocycles}\begin{description}
\item[(a)] If $n\geq5$ is odd then $\gamma(S_n)\geq \phi(n)/2$, and if in addition
$n$ is not a prime, then
$\gamma(S_n)\geq 1+\phi(n)/2$.
\item[(b)] If $n\geq 4$ is even and $n\neq 8,$ then
$\gamma(A_n)\geq 1+\phi(n)/2$.
\end{description}
\end{proposition}

\begin{proof}
If either $G=S_n$ with $n$ odd, or $G=A_n$ with $n$ even, then it follows from Lemma~\ref{two cycles} that $\gamma(G)\geq 1+|\mathcal{U}_n|$ with $\mathcal{U}_n$ as in (\ref{Un}). Thus by Lemma~\ref{lem:Un},  $\gamma(G)\geq \phi(n)/2$. We need a little more than this.

Suppose that $G=S_n$ with $n$ odd and not a prime,
%, the square of a prime or a product of twin primes,
and let $\delta$ be a basic set generating a minimal normal covering of $G$. By  Lemma~\ref{two cycles}, $\delta\supset
\{S_k\times S_{n-k}\ :\  k\in \mathcal {U}_n\}.$ Suppose that $\delta$ contains a single additional subgroup, $M$ say. Then $M$ must contain permutations of each type not occurring in $\{S_k\times S_{n-k}\ :\  k\in \mathcal {U}_n\}.$ In particular $M$ contains  at least one permutation of each of the following types: $[n],  \ [1,n-1]$
and  $[\ell_1,\dots,\ell_r]$, with $p\mid \ell_i$ for each $i=1,\dots, r$ and $r\geq 2$,
where $p$ is the smallest prime dividing $n$.  In particular $M$ is $2$-transitive and contains a permutation of type $[p,\dots,p,2p]$. An element of the latter type is an odd permutation and so $M\ne A_n$. Thus, by Lemma~\ref{Feit}, $PGL_d(q)\leq M\leq P\Gamma L_d(q)$ on the set $\Omega$ of $n=(q^d-1)/(q-1)$ points or hyperplanes, and since $M$ contains an $(n-1)$-cycle it follows from Lemma~\ref{nminusone} that $d=2$ and $q$ is prime. Thus $n=q+1$, and
since $n$ is odd and $q$ is prime, it follows that $q=2$ and $n=3$ is prime, which is a contradiction.
Thus we conclude that, if $n$ is odd and not a prime, then $\gamma(G)\geq 2+|\mathcal{U}_n|=1+\phi(n)/2$.

Now suppose that $G=A_n$ with $n$ even. For $n=4, 6$, we have from \cite{BU} that $\gamma(A_n)=2=1+\phi(n)/2$ (while $\gamma(A_8)=2 < 1+\phi(8)/2$). Thus we may assume that $n\geq10$. Let $\delta$ be  a basic set for $G$. By Lemma \ref{two cycles} (c), $\delta$ properly contains
$\{[S_k\times S_{n-k}]\cap A_n\ :\  k\in \mathcal {U}_n\}$. Suppose that $\delta$ contains  a single additional subgroup, $M$ say. Then $M$ must contain permutations of each type not occurring in $\{[S_k\times S_{n-k}]\cap A_n\ :\  k\in \mathcal {U}_n\}.$ In particular $M$ contains  at least one permutation of each of the following types:  $[n/2,n/2],  \ [1,n-1]$  and  $[\ell_1,\dots,\ell_r]$ with each $\ell_i$ even and $r$ even. A maximal intransitive subgroup is of the form $[S_k\times
S_{n-k}]\cap A_{n}$ for some $k<n/2$, so such a subgroup contains no permutation of type $[n/2,n/2]\in M.$ It follows that $M$ is transitive.
Since $[1,n-1]\in M,$ we have that $M$ is $2$-transitive and hence
primitive. By Lemma \ref{p-permutations}(b)(c), $M$ contains no permutation of type $[2,\dots,2,4],$ or of type
     $[2,\dots,2,6],$ a contradiction. Thus
$\gamma(A_n)\geq 2+|\mathcal{U}_n|=1+\phi(n)/2.$
\end{proof}

We finish this section by showing that minimal normal coverings almost always contain a transitive component.

\begin{corollary}\label{transitive component} If $n\geq 3$ and all components of a minimal normal covering of $G=A_n$ or $S_n$ are intransitive, then $n=4, G=A_4$ and the normal covering is generated by $\delta=\{A_3, C_2\}$.
\end{corollary}

\begin{proof} Suppose that $\delta$ is a basic set generating a  minimal normal covering of $G=A_n$ or $S_n$ with all components intransitive. Since for $S_n$, and if $n$ is odd also for $A_n$, some component must contain an $n$-cycle, it follows that $G=A_n$ with $n$ even. Without loss of generality we may assume that each $H\in\delta$ is maximal by inclusion among the intransitive subgroups of $G$, that is to say each $H$ is of the form $H(i):=[S_i\times S_{n-i}]\cap A_n$ for some positive integer $i\leq n/2$. Consider types $[i,n-i]$ with $1\leq i\leq n/2$. Since $n$ is even, all permutations of these types are even and so lie in $A_n$. If $H\in \delta$ contained a permutation of type $[i,n-i]$  and also one of type $[j,n-j]$ with $1\leq i< j\leq n/2$, then we would have $H\cong H(i)\cong H(j)$, which is impossible since $i<n/2$ while $n-j\geq n/2.$  Thus $\delta$ contains distinct components for distinct types $[i,n-i]$ with $1\leq i\leq n/2$, and hence $\gamma(A_n)=|\delta|\geq n/2.$ However, by Proposition~\ref{upper}, $\gamma(A_n)\leq\left\lfloor\frac{n+4}{4}\right\rfloor$, and it follows that
$n\leq 4$. Thus $n=4$ and $G=A_4$. Since the only intransitive subgroups of $A_4$ are contained in $[S_1\times S_{3}]\cap A_4=A_3$ or in $[S_2\times S_{2}]\cap A_4=C_2,$ the only possibility is $\delta=\{A_3, C_2\}.$ It is easily checked that this set $\delta$ is a basic set.
\end{proof}

\section{
$3$-cycle decompositions and intransitive components}\label{Sym-even}
Proposition~\ref{gamma twocycles} gave a generic lower bound for $\gamma(S_n)$ when $n$ is odd, and for $\gamma(A_n)$ when $n$ is even. In this section we obtain similar general lower bounds for degrees $n$ of the opposite parity. Thus we deal with $G:=S_n$ for $n$ even with $n\geq4$, and $G:=A_n$ for $n$ odd with $n\geq 5$. Here types with three entries play an important role analogous to the role played by types with two entries in the previous section. We note that types with two entries are of no use in this new situation since no such type belongs to $A_n$ with $n$ odd, and when $n$ is even then all such types belong to $A_n$ and hence may contribute to only one basic component for $S_n$.

Recalling the definition of $a(n)$ from {\rm (\ref{a})} in Section \ref{basic number} as the minimum positive integer which does  not divide $n$, we define the following set of types:
\begin{equation}\label{T}
\mathcal{T}:=\{ [i,(a(n)-1)i,n-a(n)i]\,:\ 1\leq i<\frac{n-1}{a(n)}\ \mbox{and}\ (i,n)=1\}.
\end{equation}
Note that the minimal $n$ for which $\mathcal{T}\neq\varnothing$ is $n=5.$
In the case where $n$ is divisible by 6 and $n\geq12$, we write $n=3m$ (so $m\geq 4$ is even), and consider also the following set of types, for certain intervals $I\subseteq [1,\frac{m}{2})$:
\begin{equation}\label{Tprime}
\mathcal{T}'(I):=\{ [m-i,m-2i,m+3i]\,:\ i\in I\ \mbox{and}\ (i,n)=1\}.
\end{equation}
Since $m$ is even, $i\in I$ implies that $i\leq (m-2)/2$ and so $m-i> m-2i\geq 2$. With only two exceptions in each case, the maximal proper subgroups of $G$ containing permutations of any of these types turn out to be intransitive.

\begin{lemma} \label{three cycles} Let $n\geq 5$, let $G=S_n$ with $n$ even, or $G=A_n$ with $n$ odd, and let %$a(n)$,
$\mathcal{T}, \mathcal{T}'(I)$ be as in %{\rm (\ref{a})},
{\rm (\ref{T}), (\ref{Tprime})} where $I\subseteq [1,\frac{m}{2})$.  Suppose that $H$ is a maximal subgroup of $G$ containing a permutation of type $T$, where either $T\in\mathcal{T}$, % and $a(n)$ is  prime,
or  $n\equiv 0\pmod{6}$, $n\geq12$, and $T\in \mathcal{T}'(I)$. Then either $H$ is intransitive or one of the following holds:
\begin{description}
\item[(a)]  $n=2^e+1\geq 5$, $T=[1,1,2^e-1]\in\mathcal{T}$, $G=A_n$, and $H=P\Gamma L_2(2^e)\cap G$;
\item[(b)] $n\equiv 0\pmod{6}$, and one of:
\begin{description}
\item[(i)] $n=6$, $T=[1,3,2]\in\mathcal{T}$, $G=S_6$, and $H=S_3\wr S_2$;
 \item[(ii)] $n=18$, $m=6, i=1$, $T=[5,4,9]\in\mathcal{T}'(I)$, and $H=S_9\wr S_2$; or
\item[(iii)] $n=36$, $m=12, i=5$, $T=[7,2,27]\in\mathcal{T}'(I)$, and $H=S_9\wr S_4$.
 \end{description}
\end{description}
\end{lemma}

\begin{proof} %Suppose first that $n$ is such that $a:=a(n)$ is a prime and
Let $a:=a(n)$ and suppose first that
$\sigma\in G$ is of type $T=[i,(a-1)i,n-ai]\in\mathcal{T}$. By (\ref{T}),  $n-ai\geq 2$ so, since $a\geq 2$, we have $i < n/2.$
Let $H$ be a maximal subgroup of $G$ containing $\sigma$, and suppose that $H$ is transitive. We claim that $H\not\geq A_n$. This is true by assumption if $n$ is odd. If $n$ is even then $i$ is odd and it follows that $\sigma$ has exactly one even length cycle, namely of length $(a-1)i$ if $a$ is odd, and of length $n-ai$ if $a$ is even. Hence $\sigma$ is an odd permutation, and since $H\ne S_n$ we have $H\not\geq A_n$ in this case also. This proves the claim.

Note that, by the definition of $a$ in (\ref{a}), $a-1$ divides $n$. Thus, since $(i,n)=1$ we have $(a-1,i)=1$. Also $(i,n-ai)=(i,n)=1$, and
$$
((a-1)i,n-ai)=((a-1)i,n-i)=(a-1,n-i)=(a-1,i) = 1.
$$
Thus the permutation $\sigma^{(a-1)i}$ is an $(n-ai)$-cycle and we have $n-ai\geq2$. Suppose that $H$ is primitive. Since $H\not\geq A_n,$ by Lemmas~\ref{WI 13.8}, we get that $H$ is $(ai+1)$-fold transitive. Since $a\geq 2$, we must have $3\leq ai+1\leq 5$
(see Theorem 4.11 in \cite{CA}). We consult the classification of the $2$-transitive groups as given, for example, in Tables 7.3, 7.4 of \cite{CA}.
If $H$ is 5-transitive then $H=M_{12}$ or $M_{24}$, but these cases do not arise since $a(12)=a(24)=5$. Thus $ai+1\leq4$.  If $ai+1=4$ then $a=3$ and $n$ is even and not divisible by $3$; however there are no 4-transitive groups of such a degree $n$ not containing $A_n$. Thus $ai+1=3$, so $a=2$, $i=1$, and $n$ is odd.
Since $n$ is odd and $n\geq5$, there are no affine $3$-transitive groups. For the almost simple cases the $3$-transitive possibilities of odd degree are $H=P\Gamma L_2(2^e)\cap G$ with $n=2^e+1$, $M_{11}$ and $M_{23}.$ The first family are examples as listed in (a), $M_{11}$ is excluded as it does not contain elements of order $9$, and $M_{23}$ is excluded as it does not contain elements of order $21.$

Thus we may assume that $H$ is imprimitive, so $H=[S_b\wr S_c]\cap G$ with $n=bc,2\leq b\leq n/2$. The condition $(i,n-ai)=1$ implies that the three cycles of $\sigma$ cannot be all union of blocks and also that they cannot all meet  the same set of blocks.
Suppose that the $(n-ai)$-cycle of $\sigma$ is a union of blocks of size $b,$ while the $i$-cycle and the $(a-1)i$-cycle meet the same set of $n/b-(n-ai)/b=(a/b)i$ blocks. Then $b$ divides $(n,n-ai)=(n,ai)=(n,a)<a,$ since $a$ does not divide $n,$ which gives $a/b > 1.$ It follows that the $i$-cycle meets more than $i$ blocks, a contradiction.
If the $i$-cycle of $\sigma$ is a union of blocks of size $b,$ while the $(n-ai)$-cycle and the $(a-1)i$-cycle meet the same set of $c'=\frac {n-i}{b}$ blocks, then $c'$ divides $(n-ai,(a-1)i)=1,$ which gives the contradiction $b=n-i>n/2.$
Finally if the $(a-1)i$-cycle of $\sigma$ is a union of blocks of size $b,$ while the $(n-ai)$-cycle and the $i$-cycle meet the same set of $c'=\frac {n-(a-1)i}{b}$ blocks, then $c'$ divides $(n-ai,i)=1,$ which gives
$$
b=n-(a-1)i=(n,n-(a-1)i) = (n,(a-1)i) = (n,a-1)=a-1
$$
and thus $i=\frac{n}{a-1}-1<\frac{n-1}{a}$. This inequality is equivalent to
$n< a^2-2a+1=(a-1)^2.$ Since $n\geq 5$, this implies that $a\geq4$, and hence, by the definition of $a$, that $6 \,|\, n$. If $a=4$ we have $n<(a-1)^2=9$ and hence $n=6$,
$b=3$,  $T=[1,3,2]$, and $H=S_3\wr S_2$, as in (b)(i). Thus we may assume that $a\geq 5$, and hence, by the definition of $a$, that $12 \,|\, n$.
If $n=12$ then $a=5,\ b=4$ and $i=2,$ contradicting $(i,12)=1.$ If $n=24$ then $a=5$, $b=4$ and  $i=5,$ contradicting $i<23/5.$ Thus $n\geq 36$, and Lemmas~\ref{min-prime-even} and~\ref{a-properties} apply giving $a\leq p^0(n)\leq \sqrt{n}-1.$ It follows that $n<(\sqrt{n}-2)^2=n+4-4\sqrt{n},$ which is a contradiction.

Now suppose that  $n=3m$ with $m$ even, $m\geq4$, and consider $\sigma\in G=S_n$ of type $T=[m-i,m-2i,m+3i]\in\mathcal{T}'(I)$. Let $H$ be a maximal subgroup of $G$ containing $\sigma$, and suppose that $H$ is transitive. Here $\sigma$ has exactly one even length cycle, namely of length $m-2i$ (noting that $m-2i>0$), and so $\sigma$ is an odd permutation, and in particular $H\ne A_n$.

Note also that, since $(i,n)=1$, $m$ is even and $i$ is odd, we have
\begin{equation}\label{gcd}
\begin{array}{cll}
(m-i,m-2i)&=&(m-i,i) = (m,i)=1\\
(m-i,m+3i)&=&(m-i,4i)=(m-i,i)=1\\
(m+3i,m-2i)&=&(5i,m-2i)=(5,m-2i) = 1\ \mbox{or}\ 5.\\
\end{array}
\end{equation}
Thus the permutation $\sigma^{(m-2i)(m+3i)}$ is an $(m-i)$-cycle and we have $2< m-i<n/2$. It follows from Lemma~\ref{WI 13.5} that $H$ is not primitive.
Thus $H=S_b\wr S_c$ with $n=bc, b\geq 2, c\geq2$. Suppose that the $(m-i)$-cycle of $\sigma$ is a union of blocks of size $b$. Then $b$ divides $(n,m-i)=(3m,m-i)=(3,m-i)$, and hence $b=3$ divides $m-i$. By (\ref{gcd}), $3$ does not divide $m-2i$ or $m+3i$, and so each of the remaining $\sigma$-cycles meets the same number of blocks, say $c'$, by Lemma~\ref{imprimitive}. Thus $c'$ divides $(m+3i,m-2i)$. Since $b=3 < m+3i$ we must have $c'>1$, and therefore (\ref{gcd}) implies that $c'=5$ and these $c'$ blocks comprise $15=bc'=(m+3i)+(m-2i)=2m+i$ points. Since $1\leq i<m/2$, this means that $2m<15<5m/2$, and hence $m=7$, which is a contradiction since $m$ is even. Thus the $(m-i)$-cycle of $\sigma$  is not a union of blocks. If it meets more than one block then, by Lemma~\ref{imprimitive}, it must meet the same number of blocks as the $\sigma$-cycle of length $m-2i$ or $m+3i$, and hence either $(m-i,m-2i)>1$ or $(m-i,m+3i)>1$, and neither of these holds by (\ref{gcd}). Hence the $(m-i)$-cycle of $\sigma$ is a proper subset of a block, and this block must contain exactly one of the other $\sigma$-cycles. Hence $b=(m-i)+(m-2i)=2m-3i$ or $b=(m-i)+(m+3i)=2m+2i$. Since $b\leq n/2$ it follows that $b=2m-3i$, and the $(m+3i)$-cycle is a union of blocks. Hence
\[
b=2m-3i = (2m-3i,m+3i) = (9i,m+3i) = (9,m+3i)\ \mbox{which divides 9}.
\]
Since some block is a union of two $\sigma$-cycles, if $b=3$ we would have $m-i=2, m-2i=1$ and hence $m=3$, which is a contradiction since $m$ is even. Hence $b=9$, and since $i=(2m-9)/3$ is an integer, it follows that 6 divides $m$. Moreover since $1\leq i<m/2$, we have $6\leq m<18$, whence either $m=6$, $i=1$, or $m=12, i=5$, and we have identified the two solutions in part (b).
\end{proof}

Our next result is most easily expressed using the notation introduced in Lemma~\ref{asymptoticI}, namely, for an interval $I\subseteq [0,n]$,
\[
\phi(I;n) =  \left|\left\{i\in \mathbb N^*\ :\ i\in I,\ \mbox{and}\ (i,n)=1\right\}\right|.
 \]
Recall that, if the interval $I$ has length $|I|\sim cn$, for some constant $c$, then by Lemma~\ref{asymptoticI}, $\phi(I;n) \sim c \phi(n)$.

\begin{proposition} \label{even-sym}
\begin{description}
\item[(a)] Let $G=S_n$ with $n$ even, $n\geq 4$, and not divisible by $3$, or let $G=A_n$ with $n$ odd, $n\geq 5$. Let $a=3$ if $n$ is even, and $a=2$ if $n$ is odd.
Let $I$ be the interval $(1,\frac{n-1}{a})=(1,2^{e-1})$ if $n = 2^e+1$ for some $e$, and otherwise let $I=[1,\frac{n-1}{a})$. Then
\[
\gamma(G) \geq \frac{\phi(I;n)}{2}+1\ \sim\ \frac{\phi(n)}{2a}.
\]
\item[(b)] Let $n$ be divisible by $6$, and let $I=[1,\frac{n}{9})$ if $9$ does not divide $n$, and $I=[\frac{n}{18},\frac{n}{6})$ if $9$ divides $n$. Then
\[
\gamma(S_n)\geq \phi(I;n)+1\ \sim \frac{\phi(n)}{9}.
\]
 \end{description}
\end{proposition}

%Let $n\geq 10$ be even. Then
%$$
%\gamma(S_n)\geq \displaystyle{\frac {1 }{2}\ \left|\left\{i\in \mathbf N\ :\ 1\leq i<\displaystyle{\frac{n}{6}},\ (i,n)=1\right\}\right|+1}\sim  \displaystyle{\frac{\phi(n)}{12}}.
%$$ Moreover:
%\begin{description}
%\item[(a)]If $3\nmid n$ then $$\gamma(S_n)\geq \displaystyle{\frac {1 }{2}\ \left|\left\{i\in \mathbf N\ :\ 1\leq i<\displaystyle{\frac{n-1}{3}},\  (i,n)=1\right\}\right|}+1\sim\displaystyle{\frac{\phi(n)}{6}};$$
%\item[(b)] If $3\mid n$ but $9\nmid n,$ then $$\gamma(S_n)\geq \left|\left\{i\in \mathbf N\ : \  1\leq i<\displaystyle{\frac{n}{9}},\  (i,n)=1\right\}\right|+1\sim\displaystyle{\frac{\phi(n)}{9}}.$$
%\end{description}

\begin{proof}
(a) For the values of $n$ in this case, we note that the parameter $a$ is equal to
$a(n)$, as defined in {\rm (\ref{a})}. If $n\leq 6$ the quantity $\phi(I;n)$ is zero, and $\phi(I;7)=1$ so the asserted lower bound for $\gamma(G)$ holds in these cases. Thus we may assume that $n\geq8$.

Let $\delta$ be a basic set, with maximal components, generating a minimal normal covering of $G$. Let $i\in I$ such that $(i,n)=1$, (so $i>1$ if $n=2^e+1$ for some $e$). Then by Lemma~\ref{three cycles}, $\delta$ contains an intransitive subgroup, say $H(i)$, containing a permutation of type $T_i:=[i,(a-1)i,n-ai]$.

Let $U\in\delta$ be intransitive. We claim that at most two of the types $T_i$, with $i\in I$ and $(i,n)=1$, belong to $U=[S_c\times S_{n-c}]\cap G$ for some $c$ such that $1\leq c<n/2$. If $T_i$ belongs to $U$ then $\{c,n-c\}=\{i,n-i\},$ or $\{ai,n-ai\}$, or $n$ is even, $a=3$ and $\{c,n-c\}=\{2i,n-2i\}$.
\begin{description}
\item[(i)] If  $\{c,n-c\}=\{i,n-i\},$ then $c=i\in I$, since $n-i>n/2$, and hence $(c,n)=1$.
\item[(ii)] If $a=3$ and $\{c,n-c\}=\{2i,n-2i\},$ then either $c=2i$ or $c=n-2i$, and for both possibilities $(c,n)=2$.
\item[(iii)]  If  $\{c,n-c\}=\{ai,n-ai\},$ then either $c=ai$ or $c=n-ai$, and for both possibilities $(c,n)=1$; also we have $a\,\mid c$ if $c=ai$ while $(c,a)=1$ if $c=n-ai$.
 \end{description}
Suppose first that $n$ is odd and hence $a=2.$ If $c$ is even, then either $c=i\in I$ in case (i) or $c=2i$ in case (iii), and so at most two of these types belongs to $U$ (with two occurring only if $c$ and $c/2$ both lie in $I$). If $c$ is odd, then $c=i$ or $c=n-2i$ are the only two possibilities (with both occurring only if $c$ and $(n-c)/2$ both lie in $I$). Now suppose that $n$ even with $a=3$. If $(c,n)>1$ then there are at most two possibilities, namely the ones in case (ii). If $(c,n)=1$ then again there are at most two possibilities: $c=i$ and either $c=3i$ (if 3 divides $c$) or $c=n-3i$ (if 3 does not divide $c$). Thus the claim is proved.

Hence, in order to cover all the types $T_i$ with $i\in I$ and $(i,n)=1$, we need at least $\phi(I;n)/2$ intransitive basic components in $\delta$. Since we also need a transitive component to contain an $n$-cycle, we have $|\delta|\geq 1+ \phi(I;n)/2$. That this lower bound satisfies the asymptotic property $1+\phi(I;n)/2\ \sim\ \phi(n)/2a$ follows from Lemma~\ref{asymptoticI}.

(b) Suppose now that $n=3m$ with $m$ even and $m>2$, and that $I=[1,m/3)$ if $m$ is not divisible by $3$, and $I=[m/6\,,\,m/2)$ if 3 divides $m$. Let $\delta$ be a basic set, with maximal components, generating a minimal normal covering of $S_n.$ Let $i\in I$ such that $(i,n)=1$. Then by Lemma~\ref{three cycles}, $\delta$ has a subgroup, say $H(i)$, containing a permutation of type $T_i:=[m-i,m-2i,m+3i]$, and either $H(i)$ is intransitive or $m, i, T_i,H(i)$ are as in parts (b)(ii) or (iii) of Lemma~\ref{three cycles}. We deal with the exceptional cases first. If $m=6$ then $I=[1,3)$ and $\phi(I;18)=1$; and if $m=12$ then $I=[2,6)$ and again $\phi(I;36)=1$. The inequality $\gamma(S_n)\geq 2$ is clearly true in both cases. Thus we may assume that $n\ne 18, 36$ and hence that $H(i)$ is intransitive.

Let $U$ be an intransitive subgroup in $\delta$, so $U=S_c\times S_{n-c}$ with $1\leq c<n/2=3m/2$. We claim that at most one of the types $T_i$ in the previous paragraph belongs to $U$. If $T_i$ belongs to $U$ then  $\{c,n-c\}=\{m-i,2m+i\},$ or $\{m-2i,2m+2i\}$, or $\{m+3i,2m-3i\}$. We consider these possibilities first for $i$ in $[1,m/2)$ (the union of the two possible intervals $I$) and $(i,n)=1$. Note that $m$ is even, $i$ is odd and $(3,i)=1.$
\begin{description}
\item[(i)] If  $\{c,n-c\}=\{m-i,2m+i\},$ then $c=m-i$ since $2m+i>n/2$, and hence $c$ is odd and  $m/2 <c<m$.
\item[(ii)] If $\{c,n-c\}=\{m-2i,2m+2i\},$ then $c=m-2i$ since $2m+2i>n/2$, and hence $c$ is even.
\item[(iii)]  If  $\{c,n-c\}=\{m+3i,2m-3i\},$ then either $c=m+3i$ or $c=2m-3i$, and for both possibilities $c$ is odd.
 \end{description}
If $c$ is even then at most one type $T_i$ is covered, namely $i=(m-c)/2$. So assume that $c$ is odd which implies that case (ii) does not arise. Consider first the case where 3 divides $m$. Then $I=[m/6\,,\,m/2)$, and we have $m+3i\geq 3m/2>c$ so in case (iii) only the second $c$-value may arise. Thus there is at most one type, as we can only have $c=2m-3i$ if 3 divides $c$, and  we can only have $c=m-i$ if 3 does not divide $c$. Finally assume that $(m,3)=1$ so that $I=[1,m/3)$. Then $c<m$ in case (i) and $c>m$ for either of the possible $c$-values in case (iii). Thus there is at most one type if $c<m$.  If $c>m$ then again we have at most one possibility: $c=m+3i$ if $c\equiv m\pmod{3}$, and $c=2m-3i$ if $c\equiv 2m\pmod{3}$ (and since $3\nmid m,$ the relations $c\equiv m\pmod{3}$ and $c\equiv 2m\pmod{3}$ cannot be both true). This completes the proof of the claim.

Hence, in order to cover all the types $T_i$ with $i\in I$ and $(i,n)=1$, we need at least $\phi(I;n)$ intransitive basic components in $\delta$. Since we also need a transitive component to contain an $n$-cycle, we have $|\delta|\geq \phi(I;n)+1$, and
by Lemma~\ref{asymptoticI}, this lower bound satisfies $\phi(I;n)+1\ \sim\ \phi(n)/9$.
\end{proof}

We immediately obtain the following corollary to Proposition~\ref{even-sym}.

\begin{corollary}\label{cor-lower}
If $n$ is even, not divisible by $3$, and $n\geq 4$, then setting $I=[1,\frac{n-1}{3})$,
\[
\gamma(S_n) \geq \frac{\phi(I;n)}{2}+1\ \sim\ \frac{\phi(n)}{6}.
\]
If $n$ is odd and $n\geq5$, then setting
$I'=(1,2^{e-1})$ if $n = 2^e+1$ for some $e$, and otherwise setting $I'=[1,\frac{n-1}{2})$, we have
\[
\gamma(A_n) \geq \frac{\phi(I';n)}{2}+1 = \left\{\begin{array}{ll}  \frac{\phi(n)}{4}+\frac{1}{2}
                                               &\mbox{if $n\ne 2^e+1$ for any $e$}\\
\frac{\phi(n)}{4}
                                               &\mbox{if $n= 2^e+1$ for some $e$.}\\
                                              \end{array}\right.
\]
\end{corollary}

\begin{proof}
Everything follows immediately from  Proposition~\ref{even-sym} except for the value of
$\phi(I';n)$. Note that $i\in [1,\frac{n-1}{2})$ with $(i,n)=1$ if and only if $n-i\in(\frac{n+1}{2},n-1]$ with $(n-i,n)=1$. Since $(n\pm 1)/2$ are both coprime to $n$, it follows that, if $n\ne 2^e+1$ for any $e$, then $\phi(n)=2\phi(I';n)+2$. If $n=2^e+1$, then $I'=(1,2^{e-1})$, and since also 1 and $n-1$ are coprime to $n$ we have
$\phi(n)=2\phi(I';n)+4$ in this case.
\end{proof}

While the lower bound in Corollary~\ref{cor-lower} is useful asymptotically, the result points to the possibly exceptional nature of integers of the form $n=2^e+1$.
For such integers the subgroups $P\Gamma L_2(2^e)$ may play a significant role for  normal coverings of $A_n$, especially for small $n$. By inspecting the ATLAS \cite{AT} and using $\gamma(A_9)\neq 2$ from \cite {BU}, we have the following.

\begin{example}\label{ex:a9}
$A_{9}$ admits a minimal normal covering
with basic set
$$
\{[S_{4}\times S_{5}]\cap A_9, H, K\}\quad\mbox{where}\quad H\cong K\cong P\Gamma L_2(8).
$$
In particular $\gamma(A_{9})=3.$
\end{example}
Note that we need both the two non-conjugate copies of $P\Gamma L_2(8)$ in $A_9$ because each of them meets just one of the two $A_9$-conjugacy classes of $9$-cycles (see \cite[p.37]{AT}).

\section{Minimal normal coverings if at most two primes divide $n$}

We are now in position for a first explicit computation of the gamma function. Our first result deals with $n$ prime.

\begin{proposition}\label{sym prime}
For any prime $p\geq 5$ the group $S_p$ admits a unique minimal normal covering generated by the basic set
$$
\delta=\{AGL_1(p)\cong C_p\rtimes C_{p-1}, \
S_k\times S_{p-k}\  :\ 2\leq k\leq \frac{p-1}{2}\}.
$$
In particular $\gamma(S_p)=\frac{p-1}{2}.$
\end{proposition}

\begin{proof} By Proposition \ref{gamma twocycles},
$\gamma(S_p)\geq \frac{\phi(p)}{2}=\frac{p-1}{2}.$ We claim that the set $\delta$
in the statement generates a normal covering of $S_p$. Note that, since $p\geq 5$, the subgroup $AGL_1(p)$ is a proper subgroup. Since $|\delta|=(p-1)/2$, it will follow from this claim that $\gamma(S_p)= \frac{\phi(p)}{2}=\frac{p-1}{2}.$
First we observe that $AGL_1(p)$ is a maximal subgroup of $S_p$ and contains both
   a $p$-cycle and a $(p-1)$-cycle.
Consider now a type $L=[\ell_1,\dots,\ell_r]$ with $r\geq 2,$ $1\leq \ell_1\leq \dots\leq \ell_r$ and
$\sum_1^r \ell_i=p.$ Suppose first that $r=2$. Since $L=[1,p-1]\in AGL_1(p)$, we may assume that $\ell_1\geq2$, and thus $2\leq \ell_1<p/2$. This means that $L\in S_{\ell_1}\times S_{p-l_1}$ which lies in $\delta.$ Now suppose that $r\geq 3.$ If $\ell_1\neq 1,$ then $\ell_1\leq p/3<p/2$ and hence $L\in S_{\ell_1}\times S_{p-\ell_1}$ which lies in $\delta$. Suppose then that
$\ell_1=1$, so that $\ell_2\leq(p-1)/(r-1)\leq(p-1)/2$. Thus if $\ell_2\geq 2,$ then $L\in S_{\ell_2}\times S_{p-\ell_2}$  which lies in $\delta$. On the other hand if $\ell_2=1,$ then $L=[1,1,\ell_3,\dots,\ell_r]\in S_2\times S_{p-2}$, which lies in $\delta$.

Now we prove the uniqueness of the normal covering. Let $\mu$ be any minimal basic set for $S_p.$ Then $|\mu|= (p-1)/2$ and, by Lemma \ref{two cycles},
$$
\mu=\{H,\ S_k\times S_{p-k}\ :\ 2\leq k\leq (p-1)/2\}
$$
for some proper subgroup $H$ which must contain the types $[p],\ [1,p-1].$
In particular $H$ contains an odd permutation so $H\neq A_n$.
It follows that $H$ is $2$-transitive and contains a regular cycle which, by Lemma \ref {Feit}, gives the two choices $H=AGL_1(p)$ or $PGL_d(q)\leq H\leq P\Gamma L_d(q)$ with $p=\frac{q^d-1}{q-1}.$ In the second case, by Lemma \ref{nminusone}, we have $d=2,$ so $p=q+1$, and either $q$ is prime or $(q,H)=(4,P\Gamma L_2(4))$. The former case is impossible as $q=p-1\geq4$ and is even so cannot be prime, and the latter case is impossible since it would imply $H=G$. It follows that $H=AGL_1(p)$ and  $\mu=\delta.$
\end{proof}

\begin{remark}{\rm
We do not know the exact value of $\gamma(A_p)$ for $p$ prime. Proposition~\ref{upper} gives an upper bound  $\gamma(A_{p})\leq \left\lfloor\frac{p+3}{3}\right\rfloor$, while
Corollary~\ref{cor-lower} gives the lower bound  $\gamma(A_p)\geq\lceil\frac{p+1}{4}\rceil$ when $p$ is not a Fermat prime and  $\gamma(A_p)\geq\lceil\frac{p-1}{4}\rceil=2^{e-2}$ when $p=2^e+1$ is a Fermat prime.
An analysis of $A_{11}$ shows that the upper bound is sometimes sharp, but in other cases the upper bound is not attained, the smallest example being  $\gamma(A_7)=2=\lceil\frac{8}{4}\rceil<\left\lfloor\frac{10}{3}\right\rfloor$ (see Table~\ref{tbl:exact}). In fact for $A_7$ the lower bound is attained.
}
\end{remark}

\begin{example}\label{ex:a11} The group
$A_{11}$ admits a minimal normal covering with basic set
\[
 \delta=\{ A_{10},\ [S_{2}\times S_{9}]\cap A_{11},\ [S_{3}\times S_{8}]\cap A_{11},\ M_{11}\}.
\]
 In particular $\gamma(A_{11})=4$.
\end{example}

\begin{proof}
Let $\delta$ be a basic set for a minimal normal covering of $A_{11}.$ Without loss of generality we may assume that all subgroups in $\delta$ are maximal in $A_{11}$. Since the two conjugacy classes of subgroups $M_{11}$ are the only transitive maximal subgroups, and since some basic component $M$ contains an $11$-cycle, we may assume that $M=M_{11}\in\delta$ (for one of the $M_{11}$-classes). Also, by Lemma \ref{three cycles} each of the following types belongs to an intransitive basic component:
$
[1,1,9],\ [2,2,7],\ [3,3,5],\ [4,4,3].
$
By Corollary~\ref{cor-lower} and Proposition \ref{upper}, $|\delta|=3$ or $4$ and we must prove it is 4.

Suppose to the contrary that $|\delta|=3$.
If $[S_{2}\times S_{9}]\cap A_{11}\notin\delta,$ then the only intransitive maximal subgroups containing types $[1,1,9]$ and $[2,2,7]$ are $A_{10}$ and $[S_{4}\times S_{7}]\cap A_{11}$ respectively, so we may assume that both lie in $\delta$ and the remaining component is $M$. However then no component contains the type $[3,3,5].$ Thus $M, [S_{2}\times S_{9}]\cap A_{11}\in\delta$. Similarly if $[S_{3}\times S_{8}]\cap A_{11}\notin\delta,$ then the only intransitive basic component containing $[3,3,5]$ is $[S_{5}\times S_{6}]\cap A_{11}$ so this subgroup lies in $\delta$; but then there is no component containing $[4,4,3].$ Thus
$$
\delta=\{[S_{2}\times S_{9}]\cap A_{11},\ [S_{3}\times S_{8}]\cap A_{11},\ M_{11}\}.
$$
Since $M_{11}$ does not contain elements of order $12,$ no component in $\delta$ contains a permutation of type $[1,4,6]$, a contradiction.

Thus $\gamma (A_{11})=4$ and the normal covering described in Proposition \ref{upper} with basic set
$\delta=\{A_{10},\ [S_{2}\times S_{9}]\cap A_{11},\ [S_{3}\times S_{8}]\cap A_{11},\ M_{11}\}$ is a minimal normal covering for $A_{11}.$\medskip
\end{proof}

The case of prime degree allows us to understand that the difference between $\gamma(S_n)$ and $\gamma(A_n)$ is unbounded.
\begin{corollary}\label{distance} $$\displaystyle{\limsup_{n\rightarrow +\infty}\, [\gamma(S_n)-\gamma(A_n)]=+\infty.}$$

\end{corollary}
\begin{proof} Let $p$ be a prime, with $p\geq 5$. Then, by Proposition \ref{sym prime} we have $\gamma(S_p)=(p-1)/2$ while by Proposition \ref{upper}, $\gamma(A_p)\leq (p+3)/3.$ It follows that $\gamma(S_p)-\gamma(A_p)\geq (p-1)/2-(p+3)/3= (p-9)/6$ which tends to $+\infty$ as $p$ tends to $+\infty.$
\end{proof}

Next we consider prime powers.

\begin{proposition}\label{prime power} Let $n=p^{\,\alpha},$ with $p$
prime and $\alpha\geq 2$, and let
\begin{eqnarray*}
\delta&=&\{S_p\wr S_{p^{\,\alpha -1}},\ S_k\times S_{p^{\,\alpha}-k}\
 :\ 1\leq k<\frac {p^{\,\alpha}}{2},\  p\nmid k\}\quad\mbox{and} \\
\delta_{A_n}&=&\{[S_p\wr S_{p^{\,\alpha -1}}]\cap A_n,\ [S_k\times S_{p^{\,\alpha}-k}]\cap A_n\
 :\ 1\leq k<\frac {p^{\,\alpha}}{2},\  p\nmid k\}.
\end{eqnarray*}
\begin{description}
\item[(a)] The sets $\delta, \delta_{A_n}$ are basic sets for
 $S_{n}, A_n$ respectively; $\delta$ is minimal if $p$ is odd, while $\delta_{A_n}$ is minimal if $p=2$ and $n\ne8$.
\item[(b)] $\gamma(S_{\,p^{\alpha}})=\frac{\phi(p^{\alpha})}{2}+1$ if $p$ is odd, and
$\frac{2^{\alpha-2}+2}{3}\leq \gamma(S_{\,2^{\alpha}})\leq 2^{\alpha-2}+1$.
\item[(c)]  $\gamma(A_{\,p^{\alpha}})\leq \frac{\phi(p^{\alpha})}{2}+1$ with equality if
$p=2$ and $n\ne8$; and $\gamma(A_8)=2$.
\end{description}
\end{proposition}

\begin{proof}  First we consider the set $\delta$. The group $H=S_p\wr S_{p^{\,\alpha -1}}$ contains a $p^{\,\alpha}$-cycle as well as permutations of all types $[\ell_1,\dots,\ell_r]$ for which each $\ell_i$ is divisible by $p.$ Next consider a type $[k,n-k]$ with $k\leq n/2$ and $p\nmid k$. Note that we cannot have
$k=n-k=n/2$ since this would imply $n, k, p$ all even and hence $p=2$ would divide $k$. Hence $[k,n-k]$  belongs to $S_k\times S_{p^{\,\alpha}-k}\in \delta.$ Now consider a  type $L=[\ell_1,\dots,\ell_r]$ with $r\geq 3,\ \sum_{i=1}^r \ell_i=p^{\,\alpha}$, each $\ell_i\geq 1$, and say $p\nmid \ell_1.$  Since $\ell_1=p^{\,\alpha}-\sum_{i=2}^r \ell_i$ there is a second length, say $\ell_2$, coprime to $p.$ Since $\ell_1 +\ell_2<n,$ at least one of $\ell_1$ and $\ell_2$, say $\ell_1$, is less than $n/2$.
Therefore $L\in S_{\ell_1}\times S_{p^{\,\alpha}-\ell_1}$ in
   $\delta.$  Thus $\delta$ generates a normal covering for $S_{\,p^\alpha}$ and so
$\gamma(S_{\,p^{\alpha}})\leq\frac{\phi(p^{\alpha})}{2}+1$. If $p$ is odd then equality holds by Proposition~\ref{gamma twocycles}(a). If $p=2$ and $\alpha \geq 3,$ then by Proposition~\ref{even-sym}(a), $\gamma(S_{2^\alpha})\geq 1+\phi(I;2^\alpha)/2$, where $I=[1,\frac{2^\alpha-1}{3})$. Now $\phi(I;2^\alpha)$ is the number of odd positive integers less than $\frac{2^\alpha-1}{3}$, and this number is at least $\frac{2^\alpha-4}{6}$. It follows that $\gamma(S_{2^\alpha})\geq \frac{2^{\alpha-2}+2}{3},$ which is trivially true also when $\alpha=2.$
The assertions in (b) are now proved.

Since $\delta$  does not involve $A_n$, by intersection, we get a normal covering for $A_n$, so $\gamma(A_{n})\leq 1+\phi(n)/2$. If $n=2^\alpha\geq 16$ then  $\delta_{A_n}$ is minimal by Proposition \ref{gamma twocycles}(b).  This is also true for $n=4$ but fails for $n=8$ since $\gamma(A_8)=2$ (see \cite{BU}).
\end{proof}

Observe that, by Proposition \ref{prime power}(c), it follows that the
upper bound $\gamma(A_n)\leq \lfloor\frac{n+4}{4}\rfloor$ given in Proposition \ref{upper} cannot be improved for general even $n$. Also, as noted in the proof above, the equality
$\gamma(A_{\,2^\alpha})=2^{\alpha-2}+1$
fails when $\alpha=3.$ In that case
$$
\{AGL_3(2),\ [S_3\times S_5]\cap A_8\}
$$
is a basic set for $A_8$. There are two non-conjugate copies of $AGL_3(2)$ in $A_8$ and either can be used to get a normal covering.

We now move on to the case where $n$ is a product of two distinct primes.

\begin{proposition}\label{two primes} Let $n=pq$ with $p, q$
primes and $p < q$, and let
$$
\delta(A)=\{A,\   S_k\times S_{n-k}\
 :\ 1\leq k<n/2,\  p,\ q \nmid k\}
$$
where $A=S_p\wr S_{q}$ or $A=S_q\wr S_{p}$.
\begin{description}
\item[(a)] Then, for each $A$, $\delta(A)$ is a basic set for $S_{pq}$; and $\delta(A)$ is minimal if $p$ is odd. Also by intersection, each $\delta(A)$ gives a basic set for $A_{pq}$ which is minimal if $p=2$.
\item[(b)] $\gamma(S_{pq})=\frac{\phi(pq)}{2}+1$ if $p$ is odd, and $\gamma(S_{2q})\leq\frac{\phi(q)}{2}+1 = \frac{q+1}{2}$.
\item[(c)]  $\gamma(A_{2q})= \frac{q+1}{2}$, and $\gamma(A_{pq})\leq \frac{\phi(pq)}{2}+1$ if $p$ is odd.
 \end{description}

\end{proposition}

\begin{proof} Let $\delta=\delta(A)$ with $A=S_p\wr S_{q}$ or $A=S_q\wr S_p$. Then
$A$ contains permutations of any type $[\ell_1,\dots,\ell_r]$ such that either each $\ell_i$
is divisible by $p$, or each $\ell_i$ is divisible by $q$. Consider first permutations of type $L=[k,n-k]$ with $1\leq k\leq n/2.$ If $(k,n)\neq 1,$ then $L$ belongs to $A$, since either $p$ or $q$ divides both of $k$ and $n-k$. If $(k,n)=1,$ then $k\neq n/2$ and
$L$ belongs to $S_k\times S_{n-k}\in \delta.$

Now consider a type $L=[\ell_1,\dots,\ell_r]$ with $r\geq 3$, each $\ell_i
\geq 1,\  \sum_{i=1}^r \ell_i=n$. If some $\ell_i$ is coprime to $n$, then $n-\ell_i$ is
also coprime to $n$ and some $k\in\{\ell_i,\ n-\ell_i\}$ is less than $n/2$; in this case $L$ belongs to  $S_k\times S_{n-k}\in \delta.$
Assume now that $(\ell_i,n)\neq 1$ for all $i=1,\dots, r.$ Then for every $i$ we have $p\mid \ell_i$ or $q\mid \ell_i$, but not both since $\ell_i<n=pq.$ We claim: if $q$ does not divide $\ell_j,$ for some $j$, then $p$ divides all the $\ell_i$. To prove this, write $L$ as $[pa_1,\dots,pa_s,qb_{s+1},\dots,qb_r]$ with all $a_i, b_j\geq 1$ and with $s\geq 1.$ Then $pq=pa+qb$, where $a:=\sum_{i=1}^s a_i>0$ and $b:= \sum_{i=s+1}^r b_i$. This implies that $q\mid a$ and $p\mid b$, and we therefore have $a/q+b/p=1, $ with $a/q,\ b/p$ non-negative integers, which means that one is 0 and the other is 1. Since $a>0$ we conclude that $b=0$, that is, all entries are divisible by $p.$ Thus the claim is proved, and as we observed in the previous paragraph, each such type $L$ belongs to $A\in \delta$.

Thus $\delta$ generates a normal covering, and  in particular
$\gamma(S_{pq})\leq\frac{\phi(pq)}{2}+1.$
Since the basic set $\delta$ of $S_{pq}$ does not involve $A_{pq},$ by intersection we get also a basic set for $A_{pq}$ and hence $\gamma(A_{pq})\leq \frac{\phi(pq)}{2}+1.$
If $p$, and hence $n$ is odd, then by Proposition \ref{gamma twocycles}(a) it follows that $\gamma(S_{pq})=\frac{\phi(pq)}{2}+1$, and $\delta$ is minimal. If $p=2$, and hence $n$ is even, then we have
$\gamma(A_{2q})\leq
\lfloor\frac{2q+4}{4}\rfloor=\frac{q+1}{2}$, and the reverse inequality holds by Proposition~\ref{gamma twocycles}(b). Thus the normal covering generated by $\delta$ by intersection is minimal.
\end{proof}

The upper bound for $\gamma(S_{2p})$ given in Proposition~\ref{two primes}(b) cannot be improved in general as the following example shows.

\begin{example}\label{10} The normal covering of $S_{10}$
with basic set
$$
\delta:=\{S_{2}\wr S_{5},\  \ S_3\times S_{7}, \ S_9 \}
$$
is minimal normal and produces, by intersection, a minimal normal
covering for $A_{10}.$ In particular $\gamma(S_{10})=\gamma(A_{10})=3.$
\end{example}

\begin{proof}  By Proposition~\ref{gamma twocycles}(a),  $\gamma(S_{10})\geq 3$ and by
Proposition \ref{two primes}(b) we have also  $\gamma(S_{10})\leq 3.$
Thus $\gamma(S_{10})=3$ and it is easily checked that the given set $\delta$ generates a normal covering of $S_{10}$. Thus $\delta$ is minimal, and as its basic
components are maximal and do not involve $A_{10}$, it follows that its intersection with
$A_{10}$ produces a normal covering of $A_{10}$. That this normal covering is minimal follows from Proposition~\ref{gamma twocycles}(b).
\end{proof}

Finally we consider the general case in which $n$ has just two prime divisors, but is not a product of two primes.

\begin{proposition}\label{two primes power} Let $n=p^{\alpha}q^{\beta}$ with $p, q$  primes such that $p<q$, and with $\alpha,\ \beta$ positive integers such that $(\alpha,\beta)\neq (1,1).$  Then
\begin{description}
 \item[(a)]
 $
\delta=\{S_p\wr S_{n/p},\  S_q\wr S_{n/q},\  S_k\times S_{n-k}\
 :\ 1\leq k<n/2,\  p \nmid k,\ q \nmid k \}
$
generates a normal covering for $S_{n}$, which is minimal if $n$ is odd. The intersection of $\delta$ with $A_n$ generates a normal covering of $A_n$, which is minimal if $n$ is even.
\item[(b)]  $ \gamma(S_{p^{\alpha}q^{\beta}})\leq \frac{\phi(p^{\alpha}q^{\beta})}{2}+2$, with equality if $p$ is odd.
\item[(c)] $\gamma(A_{p^{\alpha}q^{\beta}})\leq\frac{\phi(p^{\alpha}q^{\beta})}{2}+2$, with equality if $p=2$.
 \end{description}
\end{proposition}

\begin{proof}
First we show that $\delta$ is a basic set for $S_n$. We consider all types $L=[\ell_1,\dots,\ell_r]$ with each $\ell_i\geq1$, $r\geq1$, and $\sum_{i=1}^r\ell_i=n$. If each $\ell_i$ is divisible by $p$ then $L\in S_p\wr S_{n/p}$ in $\delta$, and if each $\ell_i$ is divisible by $q$ then $L\in S_q\wr S_{n/q}$ in $\delta$. In particular $\delta$ covers the $n$-cycles. So we may take $r\geq 2$. If some $\ell_i$ is coprime to $n,$ then $\ell_i\neq n/2$ and $n-\ell_i$ is coprime to $n,$ so $L\in S_{\ell_i}\times S_{n-\ell_i}$ in $\delta$. Thus we may assume that $(n,\ell_i)>1$ for all $i$, and that there are distinct $\ell_i,\, \ell_j$ such that $p\nmid\ \ell_i$ and $q \nmid\,\ell_j$. This implies that $p \mid \ell_j$ and $q \mid \ell_i$, and hence that $\ell_i+\ell_j$ is not divisible by either $p$ or $q$, so that  $(\ell_i+\ell_j,n)=1$. In particular $\ell_i+\ell_j\ne n/2$, so $L\in S_{\ell_i+\ell_j}\times S_{n-\ell_i-\ell_j}$ in $\delta$. Therefore $\delta$ is a basic set for $S_n$, and since $\delta$ does not involve $A_n,$ by Lemma~\ref{SA}, the intersection of $\delta$ with $A_n$ is a basic set for $A_n$.
Thus $ \gamma(S_{n})\leq \frac{\phi(n)}{2}+2$, and $\gamma(A_{n})\leq\frac{\phi(n)}{2}+2$.

\medskip\noindent
{\it Case: $S_n$ with $p$ odd:}\quad  By Proposition \ref{gamma twocycles}(a), $\gamma(S_n)\geq \frac{\phi(n)}{2}+1.$ Suppose for a contradiction, that $\gamma(S_n)= \frac{\phi(n)}{2}+1$, and let $\delta'$ be a minimal basic set for $S_n$.
For $1\leq k\leq n/2$ set $B_k:= S_k\times S_{n-k}$. It follows from Lemma~\ref{two cycles} that $\delta'$ contains $\{B_k\,|\ 1<k<n/2,\ p \nmid k,\ q \nmid k\}$. Let $M, K$ be the other two members of $\delta'$. We may assume that $M$ and $K$ are both maximal subgroups of $S_n$. Then the types $[n], [n-1,1]$ and each type $L=[\ell_1,\dots,\ell_r]$ with either each $\ell_i$ divisible by $p$, or  each $\ell_i$ divisible by $q$, must belong to $M$ or $K$ (or both).

We claim that $A_n\not\in\delta'$. Suppose to the contrary that $M=A_n$. Since $[n-1,1]$ is an odd permutation, it must lie in $K$. Since $K$ is maximal it is $S_{n-1}$ (if intransitive) or is 2-transitive (if transitive). Now a permutation $\sigma$ of type $[p,\dots,p,2p]$ is an odd permutation and therefore does not lie in $A_n$; $\sigma$ has no fixed points and therefore does not lie in $S_{n-1}$; and by Lemma~\ref{p-permutations}(a), $\sigma$ does not lie in a primitive proper subgroup (and hence does not lie in a 2-transitive proper subgroup). This is a contradiction since the type $[p,\dots,p,2p]$ must belong to $M$ or $K$. Thus
$A_n\not\in\delta'$.

Suppose next that both $[n]$ and $[n-1,1]$ belong to the same subgroup in $\delta'$, say $M$. By Lemma~\ref{Feit} and the maximality of $M$ it follows that $M=P\Gamma L_d(s)$ acting on $n=(s^d-1)/(s-1)$ points or hyperplanes. Then by Lemma~\ref{nminusone}, $d=2$ and, since $n\neq 5,$ then $s$ is a prime. Thus $n=s+1,$ with $n$ odd, whence $n=3$, which is a contradiction. Thus we may assume that $[n]\in M$ and $[n-1,1]\in K$. In particular $M$ is transitive, and either $K=S_{n-1}$ or $K$ is 2-transitive.

Suppose that $M$ is primitive. Then by Lemma~\ref{p-permutations},   $[p,\dots,p,2p]$ does not belong to $M$, and hence must belong to $K$. This means that, on the one hand, $K$ is not primitive (by Lemma~\ref{p-permutations} again), and on the other hand $K\ne S_{n-1}$ (since the type  $[p,\dots,p,2p]$ has no fixed points). This is a contradiction, and so $M$ is imprimitive. Thus $M=S_a\wr S_b$ for some $a>1, b>1$ with $n=ab$.

We prove next that $K$ is 2-transitive. If not then $K=S_{n-1}$, so $K$ does not contain permutations of types $[p,n-p]$ or $[q,n-q]$ and so both types belong to $M$.
Then, by Lemma~\ref{imprimitive} and the facts that $(p,n-p)=p$ and $(q,n-q)=q$, it follows that $\{a,b\}=\{p,q\}$. This implies that $n=pq$, which is a contradiction.
Thus $K$ is 2-transitive.

Now consider the types $[2p,n-2p]$ and $[2q,n-2q]$. Since $n-2p, n-2q$ are odd, we have $(2p,n-2p)=p$ and $(2q,n-2q)=q$, and the argument of the previous paragraph shows that the types cannot both belong to $M$. Thus $K$ contains an element $\sigma$ of type $[2s,n-2s]$ for some $s\in\{p,q\}$. Since $(2s,n-2s)=s$, the element $\sigma^{n-2s}$ is a product of $s$ transpositions and so $\mu(K)\leq 2s$. However $\mu(K)\geq n/3$ by Lemma~\ref{min deg}, and hence $n/s\leq 6$.  This is a contradiction since $n/s$ is a product of at least two odd primes (since $\alpha+\beta\geq3$).  Hence $\gamma(S_n)= \frac{\phi(n)}{2}+2$ if $p$ is odd.

\medskip\noindent
{\it Case: $A_n$ with $p=2$:}\quad Since $\alpha+\beta\geq3$ we have $n\geq12$. Thus  by Proposition \ref{gamma twocycles}(b), $\gamma(A_n)\geq \frac{\phi(n)}{2}+1.$ Suppose for a contradiction, that $\gamma(A_n)= \frac{\phi(n)}{2}+1$, and let $\delta'$ be a minimal basic set for $A_n$.
For $1\leq k\leq n/2$ set $B_k':= [S_k\times S_{n-k}]\cap A_n$. It follows from Lemma~\ref{two cycles} that $\delta'$ contains $\{B_k'\,|\ 1<k<n/2,\ 2 \nmid k,\ q \nmid k\}$. Let $M, K$ be the other two members of $\delta'$. We may assume that $M$ and $K$ are both maximal subgroups of $A_n$. Then the type $[n-1,1]$ and each type $L=[\ell_1,\dots,\ell_r]$ with an even number of the $\ell_i$ even, and either each $\ell_i$ divisible by $q$, or  each $\ell_i$ even, must belong to $M$ or $K$ (or both).

Suppose without loss of generality that $[n-1,1]\in K$. Then, by maximality, either $K=A_{n-1}$ or $K$ is 2-transitive. Suppose first that $K=A_{n-1}$. Then   the type $[2,\dots,2,4]$ (if $4$ does not divide $n$) or  $[2,\dots,2,6]$ (if $4$ divides $n$) belongs to $M$, and it follows from Lemma~\ref{p-permutations} that $M$ is not primitive. Also the types $[2,n-2], [4,n-4]\in M$ and it follows first that $M$ must be transitive (and hence imprimitive), and then that $M=[S_2\wr S_{n/2}]\cap A_n$ or
$M=[S_{n/2}\wr S_2]\cap A_n$. Finally the type $[q,n-q]\in M$, and by Lemma~\ref{imprimitive} this is only possible if $n=2q$, which is a contradiction. Thus the group $K$ must be 2-transitive. Now $K$ is not an affine group since $n$ is not a prime power, and hence $K$ is an almost simple 2-transitive group by Burnside's Theorem \cite[Theorem 4.1B]{DM}.

Suppose next that the stabiliser $K_n$ of the point $n\in \Omega$ is primitive on the set $\{1,2,\dots,n-1\}$. Then, since $K_n$ contains an $(n-1)$-cycle and $K_n\ne A_{n-1}$, it follows from Lemma~\ref{Feit}  that one of the following holds (we use  $n=2^\alpha q^\beta$ with $\alpha+\beta\geq3$, and we use the table of almost simple $2$-transitive groups \cite[page 197]{CA}):
\begin{description}
 \item[(i)] $K_n\leq AGL_1(r)$ with $n-1=r$ prime, and hence $K=PSL_2(r)$ or $PGL_2(r)$
(by the classification of Zassenhaus groups \cite[Theorem 11.16]{HB});

\item[(ii)] $PGL_d(r)\leq K_n\leq P\Gamma L_d(r)$ with $n-1=\frac{r^d-1}{r-1}$, but in this case there is no possible 3-transitive almost simple group $K$ of even degree  $n=2^\alpha q^\beta = 1+(r^d-1)/(r-1)$ with $\alpha+\beta\geq3$;

\item[(iii)] $K=M_{11}, M_{12}$ or $M_{24}$ with $n=12, 12$ or $24$ respectively.
 \end{description}
Suppose first that $n=12$. Then by maximality, $K=M_{12}$, so $\delta'=\{B_5'=[S_5\times S_7]\cap A_{12}, M_{12}, M\}$. Now $M_{12}$ contains no element of type $[3,9]$, $[2,2,2,6]$, $[1,1,1,9]$, or $[1,1,1,1,8]$ (see \cite{AT}), and it follows that all of these types belong to $M$, as they do not belong to $B_5'$. Thus $M$ is transitive, and by Lemma~\ref{p-permutations}, $M$ must be imprimitive. The only maximal imprimitive groups containing the type $[3,9]$ are $M=[S_3\wr S_4]\cap A_{12}$ and $M=[S_4\wr S_3]\cap A_{12}$. The former does not have elements of type $[1,1,1,1,8]$, while the latter does not have elements of type $[1,1,1,9]$. Thus $n>12$.

In case (i) above, $K$ has minimum degree $n-2$, and hence, since $n>12$ and $n>2q$, $K$ does not contain the types $[4,n-4]$, $[6,n-6]$, or $[q,n-q]$. Thus all these types belong to $M$ and it follows that $M$ is transitive. By Lemma~\ref{p-permutations}, the types $[2,\dots,2,4]$ and  $[2,\dots,2,6]$ do not belong to $K$ and hence one of them belongs to $M$, and a further application of   Lemma~\ref{p-permutations} shows that $M$ is imprimitive.
The only maximal imprimitive groups containing the type $[q,n-q]$ are $M=[S_q\wr S_{n/q}]\cap A_{n}$ and $M=[S_{n/q}\wr S_q]\cap A_{n}$ (by Lemma~\ref{imprimitive}). Thus one of these groups must contain an element of type $[4,n-4]$, and this is only possible if $n=4q$. However this implies that neither of these groups has elements of type $[6,n-6]$ since $n>12$. Thus (i) does not hold for $K$.

This leaves $n=24, K=M_{24}$ in case (iii). In this case the types $[9,15]$, $[10,14]$, $[2,\dots,2,6]$ do not belong to $K$ and hence must belong to $M$. In particular $M$ is transitive, and by Lemma~\ref{p-permutations}, $M$ is imprimitive. However the only maximal imprimitive subgroup $M=[S_{a}\wr S_b]\cap A_{24}$ containing type $[9,15]$ has $\{a,b\}=\{3,8\}$ while the only one containing type   $[10,14]$ has $\{a,b\}=\{2,12\}$. Thus we reach a contradiction.

We conclude that  $K_n$ is imprimitive on $\{1,2,\dots,n-1\}$. In particular $n-1$ is not prime. Hence $n\ne 12, 18, 20$ or 24, so $n=2^\alpha q^\beta\geq 28$. Note also that, since $K$ is 2-transitive, its minimal degree $\mu(K)\geq n/3$ by Theorem~\ref{min deg}.
If the type $[2,n-2]\in K$, then $K_n$ would contain a permutation of type $[1,1,\frac{n-2}{2}, \frac{n-2}{2}]$ and this would force $K_n$ to be primitive on $\{1,2,\dots,n-1\}$, which is a contradiction. Thus $[2,n-2]\not\in K$ and so $[2,n-2]\in M$. Also  the type $[2,\dots,2,4]$ (if $4$ does not divide $n$) or  $[2,\dots,2,6]$ (if $4$ divides $n$) does not belong to $K$ by  Lemma~\ref{p-permutations}, and hence belongs to $M$ which implies (again by  Lemma~\ref{p-permutations}) that $M$ is not primitive. The only possibilities for $M$ are therefore $M=[S_{2}\times S_{n-2}]\cap A_{n}$, $[S_{n/2}\wr S_2]\cap A_{n}$ and $[S_2\wr S_{n/2}]\cap A_{n}$.

Thus for each odd integer $a<n/q$ such that $a\ne n/(2q)$, we have $n-aq\geq q>2$ and $aq\ne n/2$ and so, by Lemma~\ref{imprimitive}, the type $[aq,n-aq]$ does not lie in any of the possibilities for $M$. Hence $K$ contains an element $g$ of this type. Suppose first that $q=3$. Then $n=2^\alpha 3^\beta\geq 36$ (since $n\geq 28$) and we may choose $a=5$. Then $(5,n-5q)=1$ so $g^{n-aq}=g^{n-15}$ moves exactly 15 points, and so  $15\geq \mu(K)\geq n/3$ by Theorem~\ref{min deg}. This implies that $n=36$. By \cite[page 197]{CA} the only possibility for $K$ is $Sp_6(2)$, but this group has no element of order $n-1=35$ (see \cite{AT}). Thus $q\geq5$ and, since $n\geq 4q$ and $3$ does not divide $n$, we may choose $a=3$. In this case $(3,n-3q)=1$ so  $g^{n-aq}=g^{n-3q}$ moves exactly $3q$ points, and so  $3q\geq \mu(K)\geq n/3$ by Theorem~\ref{min deg}. This implies that $n=4q$ or $8q$ since $q\geq5$. By \cite[page 197]{CA} the possible groups are the following, and we use information about these groups from \cite{AT}:
\begin{description}
 \item[(i)] $n=4\times 7=28, K=P\Gamma L_2(8), Sp_6(2)$, or $U_3(3)$, but none of these groups contains an element of order 27;

\item[(ii)] $n=8\times 43=344, K = U_3(7)$, but $K$ contains no element of order 343;

\item[(iii)] $n=8\times 17=136, K=Sp_8(2)$, but $K$ contains no element of order 135;

\item[(iv)] $n=\frac{r^d-1}{r-1},\  PSL_d (r) \leq K\leq P\Gamma L_d(r) \cap A_n.$ Since $K$ contains an $(n-1)$-cycle, by Lemma~\ref{nminusone}  we get that $d=2$ and $r$ is a prime (since $n\ne5$). It follows  that $K\leq PGL_2(r)$ and then $\mu(K)\geq \mu(PGL_2(r))=n-2$. Therefore $3q\geq n-2\geq 4q-2$, which is a contradiction.
 \end{description}
Thus if $n$ is even then $\gamma(A_n)= \frac{\phi(n)}{2}+2$.
\end{proof}

Proposition \ref{two primes power} gives exact values of $\gamma$ for many small alternating and symmetric groups and a sample of such values is given in Example~\ref{gamma 45}. Moreover,
although Proposition \ref{two primes power} does not determine the exact $\gamma$-value in the case of symmetric groups with $n$ of the form $2^\alpha q^\beta$, it is possible to find this for some small even values of $n$. We do this for $n=12$ below.

\begin{examples} \label{gamma 45} $\gamma(S_{45})=8,\ \gamma(S_{75})=22,\ \gamma(S_{63})=11,\ \gamma(A_{18})=5,\ \gamma(A_{20})=6,\ \gamma(A_{36})=8,\  \gamma(A_{44})=12.$
\end{examples}

\begin{corollary}\label{gamma S12}
$\gamma(S_{12})=\gamma(A_{12})=4.$
\end{corollary}

\begin{proof} By Proposition \ref{two primes power}, $\gamma(A_{12})=4$ and $\gamma(S_{12})\leq 4$, and moreover, by \cite{BBH}, $\gamma(S_{12})\geq 3$. Assume now that $\gamma(S_{12})=3$, and let $\delta$ be a basic set   for a minimal normal covering. Then $\delta$  must contain $A_{12}$ as otherwise, by intersection, we would get a 3-normal covering for $A_{12}$ contradicting $\gamma(A_{12})=4$. Since $A_{12}$ does not contain $12$-cycles, we may assume that $\delta=\{A_{12}, H, K\}$ for some subgroups $K,\ H$ maximal in $S_{12}$ and $K$ transitive.

By Lemma~\ref{three cycles}, $H$ must be an intransitive subgroup containing a permutation of type $[1,4,7]$ and a permutation of type $[3,2,7],$ and hence $H=S_5\times S_7.$ Since no permutation of type $[1,2,9]$ belongs to
$A_{12}$ or to $S_5\times S_7$ we must have $[1,2,9]\in K.$ Then $\sigma^9\in K$  is a transposition and, from Lemma~\ref{WI 13.5}, it follows  that $K$ is imprimitive. The only imprimitive maximal subgroups of $S_{12}$ containing $[1,2,9]$ are conjugate to  $S_3\wr S_4$ and thus $\delta=\{A_{12},\ S_5\times S_7,\ S_3\wr S_4\}.$ However, no component in this set $\delta$ contains the type $[1,3,8].$
\end{proof}

\section{
Aknowledgements}
The first author thanks J. Sonn for suggesting the study this kind of problem. Both authors thank
S. Dolfi and  S. Vessella for many enlightening remarks
and F. Luca for number theoretic advice which helped us to determine the asymptotic behavior of some of our lower bounds.

The first author is supported by GNSAGA. The second author is supported by Australian Research Council Federation Fellowship FF0776186.

%\section{
%Table of the first $30$ values of the gamma function.}\label{table}
%

\end{document}